\documentclass[12pt,twoside,a4]{amsart}
\usepackage{amsmath,amssymb}
\usepackage{graphicx}
\usepackage{float}
 \usepackage[all]{xy}
\newtheorem{thm}{Theorem}[section]

\newtheorem{lem}[thm]{Lemma}
\newtheorem{prop}[thm]{Proposition}
\newtheorem{fact}[thm]{Fact}
\newtheorem{example}[thm]{Example}
\theoremstyle{definition}

\theoremstyle{remark}
\newtheorem{rem}[thm]{Remark}
\numberwithin{figure}{section}
\numberwithin{equation}{section}

\newcommand{\R}{\mathbb R}  
\newcommand{\Z}{\mathbb Z}  
\renewcommand{\Im}{{\mathrm{Im\,}}}  

\newcommand{\C}{\mathbb C}  
\renewcommand{\H}{\mathbb H}  

\newcommand{\rightsetse}[1]{%
\hidewidth\rotatebox[origin=c]{-45}{$\xrightarrow{\kern2em}$} 
     \rlap{\raisebox{1ex}
     {$\kern-.8em\scriptstyle #1$}}\hidewidth}
\newcommand{\rightsetsw}[1]{%
\hidewidth\rotatebox[origin=c]{45}{$\xleftarrow{\kern2em}$} 
     \rlap{\raisebox{.1ex}
     {$\kern-.8em\scriptstyle #1$}}\hidewidth}
\newcommand{\leftsetsw}[1]{%
\hidewidth
     \llap{\raisebox{1ex}
     {$\scriptstyle #1$\kern-.8em}}
    \rotatebox[origin=c]{45}{$\xleftarrow{\kern2em}$}\hidewidth}
\newcommand{\rightsetnw}[1]{%
\hidewidth\rotatebox[origin=c]{135}{$\xrightarrow{\kern2em}$} 
     \rlap{\raisebox{1ex}
     {$\kern-.8em\scriptstyle #1$}}\hidewidth}
\newcommand{\rightsetd}[1]{%
\hidewidth\rotatebox[origin=c]{-90}{$\xrightarrow{\kern2em}$} 
     \rlap{{$\scriptstyle #1$}}\hidewidth}
\usepackage[breaklinks]{hyperref}
\begin{document}
\title[Geometric analysis on small representations of $GL(N,\R)$]{Geometric analysis on small unitary representations of $GL(N,\R)$
}
\author{Toshiyuki Kobayashi, Bent \O rsted, Michael Pevzner}
\begin{abstract}
The most degenerate unitary principal series representations $\pi_{i\lambda,\delta}$ ($\lambda\in \R,\,\delta\in\mathbb Z/2\mathbb Z$)
of $G = GL(N,\R)$ attain the minimum of the Gelfand--Kirillov
dimension among all irreducible unitary representations of $G$.
This article gives an explicit formula of the irreducible
decomposition of the restriction $\pi_{i\lambda,\delta}|_H$ (\textit{branching law})
with respect to all symmetric pairs $(G,H)$.
For $N=2n$ with $n \ge 2$, the restriction
$\pi_{i\lambda,\delta}|_H$ remains irreducible for $H=Sp(n,\R)$
if $\lambda\ne0$
and splits into two irreducible representations if
$\lambda=0$.
The branching law of the restriction $\pi_{i\lambda,\delta}|_H$ is purely
discrete for $H = GL(n,\C)$,
consists only of continuous spectrum for
$H = GL(p,\R) \times GL(q,\R)$ $(p+q=N)$,
and contains
 both
discrete and continuous spectra for $H=O(p,q)$ $(p>q\ge1)$. Our emphasis is laid on geometric analysis, which arises from the restriction of `small representations' to various subgroups.
\\

\end{abstract}
\renewcommand{\thefootnote}{}
\renewcommand{\footnotemark}{}
\maketitle 
\section{Introduction}
%
The subject of our study is geometric analysis on `small
representations' of $GL(N,\R)$ through branching problems to
non-compact subgroups.

Here, by a branching problem, we mean a general question on the understanding how irreducible
representations of a group decompose when restricted to a
subgroup. A classic example
 is studying the irreducible decomposition
 of the tensor product of two representations.
Branching problems are one of the most basic problems
 in representation theory,
 however,
 it is hard in general
 to find explicit branching laws
 for unitary representations of non-compact reductive groups.
For reductive symmetric spaces
$G/H$, the multiplicities in the Plancherel formula 
 of
$L^2(G/H)$ are finite \cite{BS,Delorme},
whereas the multiplicities in the branching laws for the restriction
$G \downarrow H$ are often infinite even when $(G,H)$ are symmetric pairs
 (see \emph{e.g.} \cite{K02} for recent developments and open problems
 in this area).

Our standing point is that `\emph{small representations' of a group should have `large
symmetries' in the representation spaces}, as was advocated by one of the
authors from the
perspectives in global analysis \cite{xsato}.
In particular,
considering the restrictions of \lq small representations\rq\ to reasonable subgroups,
we expect that their breaking symmetries should have still fairly
large symmetries,
for which geometric analysis would deserve finer study.

Then, what are `small representations'?
For this, the Gelfand--Kirillov dimension serves as a coarse measure of
the `size' of infinite dimensional representations.
We recall that for an irreducible unitary representation $\pi$ of a real reductive Lie
group $G$ the Gelfand--Kirillov dimension $\mathrm{DIM}(\pi)$
 takes the value in the set of half the dimensions
 of nilpotent orbits in the Lie algebra ${\mathfrak {g}}$.
We may think of $\pi$ as one of the `smallest' infinite dimensional
 representations of $G$,
 if $\mathrm{DIM}(\pi)$ equals $n(G)$, half the dimension
 of the minimal nilpotent orbit.

For the metaplectic group $G=Mp(m,{\mathbb{R}})$,
 the connected two-fold covering group of the symplectic group $Sp(m,{\mathbb{R}})$ of rank $m$,
 the Gelfand--Kirillov dimension attains its minimum
 $n(G) = m$
 at the Segal--Shale--Weil representation.
For the indefinite orthogonal group $G=O(p,q)$
 ($p,q>3$), there exists $\pi$ such that
$\mathrm{DIM}(\pi) = n(G) \ ( = p+q-3)$
 if and only if $p+q$ is even according to an algebraic result of Howe and Vogan.
See \emph{e.g.}\ a survey paper \cite{xgansavin} for the algebraic
 theory of `minimal representations',
and \cite{Foll, xsato,KobOr2,KM-intopq}
 for their analytic aspects.

In general,
a real reductive Lie group $G$ admits at most finitely many irreducible unitary representations
 $\pi$ with $\mathrm{DIM}(\pi)=n(G)$
if
 the complexified Lie algebra ${\mathfrak {g}}_{\mathbb{C}}$
 does not contain a simple factor of type $A$ (see \cite{xgansavin}).
In contrast,
 for $G=GL(N,{\mathbb{R}})$,
 there exist infinitely many irreducible unitary representations $\pi$
 with $\mathrm{DIM}(\pi)=n(G) \ (=N-1)$.
For example,
the unitarily induced representations
\begin{equation}\label{eqn:indGL}
\pi_{i\lambda,\delta}^{GL(N,\R)} := \displaystyle{\mathrm{Ind}}_{P_N}^{GL(N,\R)}(\chi_{i\lambda,\delta})
\end{equation}
from a unitary character $\chi_{i\lambda,\delta}$
 of a maximal parabolic subgroup
\begin{equation}\label{eqn:PN}
    P_N:=(GL(1,{\mathbb{R}}) \times GL(N-1, {\mathbb{R}})) \ltimes {\mathbb{R}}^{N-1},
\end{equation}
 are such representations with parameter $\lambda\in \R$ and $\delta \in \Z/2\Z$.
\vskip10pt

In this paper,
 we find the irreducible decomposition
 of these \lq{small} representations\rq\ $\pi_{i \lambda, \delta}^{GL(N,\R)}$
 with respect to \emph{all symmetric pairs}.

We recall that a pair of Lie groups $(G,H)$ is said to be a
symmetric pair if there exists an involutive automorphism
$\sigma$ of $G$ such that $H$ is an open subgroup of 
$G^\sigma := \{g \in G: \sigma g=g \}$. 
According to M. Berger's classification \cite{Be}, the following
subgroups $H=K,\,G_j\,(1\leq j\leq 4)$ and $G$ exhaust all
symmetric pairs $(G,H)$ for $G=GL(N,\R)$ up to local isomorphisms and the center of
$G$:
 { \allowdisplaybreaks
\begin{alignat*}{3}
&K &&:= O(N) && \text{(maximal compact subgroup),}
\\
&G_1 &&:= Sp(n,\R) &&(N=2n),
\\
&G_2 &&:= GL(n,\C) &&(N=2n),
\\
&G_3 &&:= GL(p,\R) \times GL(q,\R)
\quad
&&(N=p+q),
\\
&G_4 &&:= O(p,q)   &&(N=p+q).
\end{alignat*}
}%

It turns out that the branching laws for the restrictions of
$\pi_{i\lambda,\delta}^{GL(N,\R)}$ with respect to these subgroups
behave nicely in all the cases,
and in particular, the multiplicities of irreducible representations
in the branching laws are
uniformly bounded.

To be more specific, the restriction of
$\pi_{i\lambda,\delta}^{GL(N,\R)}$ to $K$ splits discretely into the
space of spherical harmonics on $\R^N$, and the resulting $K$-type
formula is multiplicity-free and so called of ladder type.
For the non-compact subgroups $G_j\,(1\leq j\leq 4)$,
we prove the following irreducible decompositions in Theorems \ref{thm:2}, \ref{thm:3}, \ref{thm:6.1} and \ref{thm:5}:
\begin{thm}\label{thm:global}
 For $\lambda\in\R$ and $\delta\in\mathbb Z/2\mathbb Z$, the irriducible unitary representation $\pi_{i\lambda,\delta}^{GL(N,\R)}$
decomposes when restricted to symmetric pairs as follows:
\begin{enumerate}
 \item[1)] $GL(2n,\R)\downarrow Sp(n,\R)$,
$(n\geq2)$:
$$
\pi_{i\lambda,\delta}^{GL(2n,\R)}\Big|_{G_1}\simeq\left\{\begin{array}{ll}
                                                            {\mathrm{Irreducible}} & (\lambda\neq0), \\
                                                            \left(\pi_{0,\delta}^{Sp(n,\R)}\right)^+
                                                            \oplus\left(\pi_{0,\delta}^{Sp(n,\R)}\right)^-
                                                            &(\lambda=0).
                                                          \end{array}\right.
$$

  \item[2)] $GL(2n,\R) \downarrow GL(n,\C)$:
$$
\pi_{i\lambda,\delta}^{GL(2n,\R)}\Big|_{G_2}\simeq
\sideset{}{^\oplus}\sum_{m\in2\mathbb Z+\delta}
\pi_{i\lambda,m}^{GL(n,\mathbb C)}.
$$
  \item[3)] $GL(p+q,\R)\downarrow GL(p,\R)\times GL(q,\R):$
$$
\pi_{i\lambda,\delta}^{GL(p+q,\R)}\Big|_{G_3}\simeq
\sum_{\delta'\in\mathbb Z/2\mathbb Z}\int_\R^\oplus
\pi_{i\lambda',\delta'}^{GL(p,\R)}\boxtimes
\pi_{i(\lambda-\lambda'), \delta-\delta'}^{GL(q,\R)}d\lambda'.
$$
  \item[4)] $GL(p+q,\R)\downarrow O(p,q):$
 $$
\pi_{i\lambda,\delta}^{GL(p+q,\R)}\Big|_{G_4}\simeq
\sideset{}{^\oplus}\sum_{\nu \in A_+^\delta(p,q)}
\pi_{+,\nu}^{O(p,q)}\oplus \sideset{}{^\oplus}\sum_{\nu \in
A_+^\delta(q,p)}\pi_{-,\nu}^{O(p,q)}\oplus
2\int_{\R_+}^\oplus\pi_{i\nu,\delta}^{O(p,q)}d\nu.
$$
\end{enumerate}
\end{thm}
Here, each summand in the right-hand side stands for (pairwise inequivalent) irreducible
representations of the corresponding subgroups which will be defined explicitly in Sections 
\ref{sec:prf}, \ref{sec:glnc}, \ref{sec:glpq} and \ref{sec:opq}.

As indicated above, we see that the representation
$\pi_{i\lambda,\delta}^{GL(2n,\R)}$ remains generically irreducible
when restricted to the subgroup $G_1=Sp(n,\R)$
and splits into a direct sum of two irreducible subrepresentations
for $\lambda=0$ and $n>1$. The case $n=1$ is well known (\emph{cf}. \cite{B47}): 
the group $Sp(1,\mathbb{R})$ is
isomorphic to $SL(2,\mathbb{R})$, and
$\pi_{i\lambda,\delta}$ are
irreducible except for $(\lambda,\delta)=(0,1)$, while $\pi_{0,1}$
splits into the direct sum of two irreducible unitary
representations i.e.\ the (classical) Hardy space and its dual.

 The representation
$\pi_{i \lambda, \delta}^{GL(2n,\R)}$
 is discretely decomposable in the sense of \cite{xkAnn98}
when restricted to the subgroup $G_2=GL(n,\C)$. In other words, the non-compact group $G_2$ behaves in
the representation space of $\pi_{i\lambda,\delta}^{GL(2n,\R)}$ as
if it were
 a compact subgroup.
In contrast, the restriction of $\pi_{i \lambda,
\delta}^{GL(p+q,\R)}\,$
 to  another subgroup $G_3=GL(p,{\mathbb{R}})\times GL(q, {\mathbb{R}})\,$
decomposes without discrete spectrum,
while both discrete and continuous spectra appear for the
restriction of $\pi_{i\lambda, \delta}^{GL(p+q,\R)}$ to $G_4=O(p,q)$
if $p,q\ge1$ and $(p,q) \ne (1,1)$.
  Finally, in Theorem \ref{thm:MetMet} we give an irreducible
  decomposition of the tensor product of
  the Segal--Shale--Weil representation with its dual, giving another example of explicit branching laws of small representations with respect to symmetric pairs.

We have stated Theorem \ref{thm:global} from representation theoretic
 viewpoint. 
 However, our emphasis is not only on results of this nature but also on 
geometric analysis of concrete models via branching laws of  small
 representations, which we find surprisingly
rich in its interaction with  various domains of
classical analysis and their new aspects. It
 includes
the theory of Hilbert-space valued Hardy spaces (Section 2), the Weyl operator calculus (Section 3),
representation theory of Jacobi
and Heisenberg groups, the Segal--Shale--Weil
representation
of the metaplectic group (Section 4), (complex) spherical
harmonics (Section 5), the $K$-Bessel functions (Section \ref{sec:minK}),
and global analysis on space forms of indefinite-Riemannian manifolds (Section \ref{sec:opq}).

Further, we introduce a \emph{non-standard} $L^2$-model for the degenerate
principal series representations of $Sp(n,\R)$ where the Knapp--Stein
intertwining operator becomes an \emph{algebraic operator} (Theorem
\ref{thm:KNalg}). In this model the minimal $K$-types are given  in
terms of Bessel functions
(Proposition \ref{prop:FFF}). The two irreducible components
$\pi_{0,\delta}^\pm$ at $\lambda=0$ in Theorem \ref{thm:global} 1) will be
presented in three ways, that is, in terms of
Hardy spaces based on the Weyl operator calculus as giving the
$P$-module structure,
complex spherical harmonics as giving the $K$-module structure,
and the eigenspaces of the Knapp--Stein intertwining operators
 (see Theorem \ref{thm:twosum}).\vskip10pt

The authors are grateful to an anonymous referee for bringing the papers of
Barbasch \cite{Barb} and Farmer \cite{Farmer} to our attention.
\medskip

\textbf{Notation:}
$\mathbb{N} = \{0,1,2,\dotsc\}$,
$\mathbb{N}_+ = \{1,2,3,\dotsc\}$,
${\mathbb{R}}_{\pm}=\{\rho \in {\mathbb{R}}: \pm \rho \ge0\}$,
$\mathbb{R}^\times = \mathbb{R} \setminus \{0\}$, and
$\mathbb{C}^\times = \mathbb{C} \setminus \{0\}$.

\section{Hilbert space valued Hardy space}\label{sec:2}
\numberwithin{equation}{section}

Let $W$ be a (separable) Hilbert space.
Then,
we can define the Bochner integrals
of weakly measurable functions
on $\R$ with values in $W$.
For a measurable set $E$ in ${\mathbb{R}}$,
 we denote by $L^2(E,W)$ the Hilbert space
 consisting of $W$-valued square integrable functions on $E$.
Clearly, it is a closed subspace of $L^2(\R,W)$.

Suppose $F$ is a $W$-valued function
 defined on an open subset in ${\mathbb{C}}$.
We say $F$ is holomorphic
 if the scalar product $(F,w)_W$ is a holomorphic function for any $w \in W$.

Let $\Pi_+$ be the upper half plane
$\{z=t+iu\in\mathbb{C}: u=\Im z > 0\}$.
Then, the $W$-valued Hardy space is defined as
\begin{equation}\label{eqn:H}
     {\mathcal {H}}_+^2(W):=\{F : \Pi_+ \to W: 
     \text{$F$ is holomorphic and $||F||_{{\mathcal {H}}_+^2(W)}<\infty $}\},
\end{equation}
where the norm $\|F\|_{\mathcal{H}_+^2}(W)$ is given by
$$
     ||F||_{{\mathcal {H}}_+^2(W)}:=
     \left(\sup_{u >0} \int_{\mathbb{R}} ||F(t+i u)||_W^2d t\right)^{\frac 1 2}.
$$

Similarly,
${\mathcal{H}}_-^2(W)$ is defined by replacing
 $\Pi_+$ with the lower half plane $\Pi_-$.
Notice that ${\mathcal {H}}_+^2(W)$ is the classical Hardy space, if
$W={\mathbb{C}}$.\\

Next,
we define the $W$-valued Fourier transform ${\mathcal {F}}$ as
\begin{equation*}
     {\mathcal {F}}: L^2({\mathbb{R}},W) \to L^2({\mathbb{R}}, W),
\quad
                        f (t) \mapsto ({\mathcal {F}}f)(\rho) := \int_{\mathbb{R}}f(t) e^{-2\pi i \rho t}\,dt.
\end{equation*}
Here, the Bochner integral converges for $f \in (L^1 \cap
L^2)(\mathbb{R},W)$ with obvious notation.
Then, $\mathcal{F}$ extends to the Hilbert space
$L^2(\mathbb{R},W)$ as a unitary isomorphism.
\begin{example}\label{ex:FRk}
Suppose
$W=L^2(\mathbb{R}^k)$ for some $k$.
 Then, we have a natural unitary
isomorphism $L^2(\mathbb{R},W) \simeq L^2(\mathbb{R}^{k+1})$.
Via this isomorphism,
the $L^2(\R^k)$-valued Fourier transform
 $\mathcal{F}$ is identified with the partial Fourier
transform $\mathcal{F}_t$ with respect to the first variable $t$ as
follows:
\begin{equation}\label{eqn:FLk}
\begin{matrix}
  L^2(\mathbb{R},L^2(\mathbb{R}^k))
  & \underset{\mathcal{F}}{\stackrel{\sim}{\to}}
  & L^2(\mathbb{R},L^2(\mathbb{R}^k))
  \\
  |\wr && |\wr
  \\[1ex]
  L^2(\mathbb{R}^{k+1})
  & \underset{\mathcal{F}_t}{\stackrel{\sim}{\to}}
  & L^2(\mathbb{R}^{k+1}).
\end{matrix}
\end{equation}
\end{example}

As in the case of
 the classical theory on the (scalar-valued)
Hardy space $\mathcal{H}_+^2 \equiv \mathcal{H}_+^2(\mathbb{C})$,
we can characterize
 ${\mathcal {H}}_{\pm}^2 (W)$ by means of the Fourier transform:

\begin{lem}
\label{lem:Hardy}
Let $W$ be a separable Hilbert space,
and $\mathcal{H}_\pm^2(W)$ the $W$-valued Hardy spaces
(see \eqref{eqn:H}).
\begin{enumerate}
\item[1)]
For $F \in {\mathcal {H}}_{\pm}^2(W)$,
the boundary value
$$
F(t\pm i0) := \underset {u \downarrow 0}\lim F(t \pm i u)
$$
exists as a weak limit in the Hilbert space $L^2({\mathbb{R}}, W)$,
 and defines an isometric embedding:
\begin{equation}\label{eqn:HpmL}
  {\mathcal {H}}_{\pm}^2(W) \hookrightarrow L^2({\mathbb{R}}, W).
\end{equation}
From now,
 we regard ${\mathcal {H}}_{\pm}^2(W)$ as a closed subspace
 of $L^2({\mathbb{R}}, W)$.
\item[2)]
The $W$-valued Fourier transform ${\mathcal{F}}$ induces the unitary isomorphism:
$$
     {\mathcal {F}} : {\mathcal {H}}_{\pm}^2(W) \overset{\sim}{\to} L^2({\mathbb{R}}_{\pm}, W) .
$$
\item[3)]
$L^2({\mathbb{R}}, W)= {\mathcal {H}}_{+}^2(W) \oplus {\mathcal
{H}}_{-}^2(W)$ (direct sum).
\item[4)] If a function $F\in {\mathcal {H}}_{+}^2(W)$ satisfies
$F(t+i0)=F(-t+i0)$ then $F\equiv0$.
\end{enumerate}
\end{lem}
\begin{proof}
The idea is to reduce the general case to the classical one by using a
uniform estimate on norms as the imaginary part $u$ tends to zero.

Let $\{e_j\}$ be an orthonormal basis of $W$.
Suppose
$F \in \mathcal{H}_+^2(W)$.
Then we have
\begin{align}\label{eqn:Ijy}
\|F\|_{\mathcal{H}_+^2(W)}^2
&= \sup_{u>0} \int_{\mathbb{R}}
   \|F(t+iu)\|_W^2 dt
\nonumber
\\
&= \sup_{u>0} \sum_j I_j(u),
\end{align}
where we set
\[
I_j(u) := \int_{\mathbb{R}}
          |(F(t+iu),e_j)_W|^2 dt.
\]
Then, it follows from \eqref{eqn:Ijy} that  for any $j$
$\sup_{u>0} I_j(u) < \infty$
and therefore
$$
F_j(z):=(F_j(z),e_j)_W,
\quad
(z = t+iu \in \Pi_+)
$$
belongs to the (scalar-valued) Hardy space $\mathcal{H}_+^2$.
By the classical Paley--Wiener theorem for the (scalar-valued) Hardy space
$\mathcal{H}_+^2$,
we have:
\begin{align}
&\text{$F_j(t+i0) := \lim_{u\downarrow0} F_j(t+iu)$
       (weak limit in $L^2(\mathbb{R})$)},
\label{eqn:Fj}
\\
&\mathcal{F}F_j(t+i0) \in L^2(\mathbb{R}_+),
\label{eqn:FFj}
\\
&(\mathcal{F}F_j(t+iu))(\rho)
 = e^{-2\pi u\rho} (\mathcal{F}F_j(t+i0))(\rho)\,{\mathrm{for}}\, u>0,
\label{eqn:FFyj}
\\
&\text{$I_j(u)$ is a monotonely decreasing function of $u>0$},
\label{eqn:mono}
\\
&\lim_{u\downarrow0} I_j(u)
 = \|F_j(t+iu)\|_{\mathcal{H}_+^2}^2
 = \|F_j(t+i0)\|_{L^2(\mathbb{R})}^2.
\label{eqn:limIj}
\end{align}
The formula \eqref{eqn:FFyj} shows \eqref{eqn:mono},
 which is crucial in the uniform estimate as below.
In fact by \eqref{eqn:mono} we can exchange
$\sup_{u>0}$ and $\sum_j$ in \eqref{eqn:Ijy}.
Thus,
we get
\begin{equation*}
\|F\|_{\mathcal{H}_+^2(W)}^2
= \sum_j \lim_{u\downarrow0} I_j(u)
= \sum_j \|F_j(t+i0)\|_{L^2(\mathbb{R})}^2.
\end{equation*}
Hence we can define an element of $L^2(\mathbb{R},W)$ as the following weak limit:
\[
F(t+i0) := \sum_j F_j(t+i0) e_j.
\]
Equivalently,
$F(t+i0)$ is the weak limit of
$F(t+iu)$ in $L^2(\mathbb{R},W)$ as $u\to0$.
Further,
\eqref{eqn:FFj} implies
$\operatorname{supp}\mathcal{F}F(t+i0) \subset \mathbb{R}_+$ because
\[
\mathcal{F}F(t+i0)
= \sum_j \mathcal{F}F_j(t+i0)e_j
\quad \text{(weak limit)}.
\]
In summary we have shown that
$F(t+i0) \in L^2(\mathbb{R},W)$,
$\mathcal{F}F(t+i0) \in L^2(\mathbb{R}_+,W)$,
and
\[
\|F\|_{\mathcal{H}_+^2(W)}
= \|F(t+i0)\|_{L^2(\mathbb{R},W)}
= \|\mathcal{F}F(t+i0)\|_{L^2(\mathbb{R}_+,W)}
\]
for any $F \in \mathcal{H}_+^2(W)$.
Thus, we have proved that the map
\[
\mathcal{F}: \mathcal{H}_+^2(W) \to L^2(\mathbb{R}_+,W)
\]
is well-defined and isometric.

Conversely, the opposite inclusion
$\mathcal{F}^{-1}(L^2(\mathbb{R}_+,W)) \subset \mathcal{H}_+^2(W)$
is proved in a similar way.
Hence the statements 1), 2) and 3) follow.

The last statement is now immediate from 2) because
$\mathcal{F}F(t+i0)(\rho) = \mathcal{F}F(-t+i0)(-\rho)$.
\end{proof}
\vskip 10pt
\section{Weyl Operator Calculus}\label{sec:3}
In this section, based on the well-known construction of the Schr\"odinger representation
and the Segal--Shale--Weil representation, we introduce the action of
the outer automorphisms of the Heisenberg group
on the Weyl operator calculus (see \eqref{eqn:Optau},
\eqref{eqn:thetatau}, and \eqref{eqn:mettau}),
and discuss carefully its basic properties,
see Proposition \ref{fact:irred} and Lemma \ref{lem:invol}.
In particular,
the results of this section will be used in 
analyzing of the `small representation' $\pi_{i\lambda,\delta}^{GL(2n,\R)}$, when restricted to
a certain maximal parabolic subgroup of $Sp(n,\R)$, see e.g.
the identity
\eqref{eqn:Op-phi-rho}.\\

Let $\R^{2m}$ be the $2m$-dimensional Euclidean vector space endowed with the standard
symplectic form
\begin{equation}\label{eqn:symp}
\omega(X,Y)\equiv\omega((x,\xi),(y,\eta)):=\langle\xi, y\rangle-\langle x,\eta\rangle.
\end{equation}

The choice of this non-degenerate closed 2-form gives a standard realization of the symplectic
group $Sp(m,\R)$ and the Heisenberg group $H^{2m+1}$.
Namely, 
$$
Sp(m,\R):=\{T\in GL(2m,\R):\,\omega(TX,TY)=\omega(X,Y)\}
$$
 and
$$
H^{2m+1}:=\{g=(s,A)\in\R\times\R^{2m}\}
$$
 equipped with the
product
$$
g\cdot g' \equiv (s,A)\cdot(s',A'):=(s+s'+\frac12\omega(A,A'),A+A').
$$

Accordingly, the Heisenberg Lie algebra $\mathfrak h^{2m+1}$ is then defined by
$$
[(s,X),(t,Y)]=(\omega(X,Y),0).
$$

Finally we denote by $Z$ the center $\{(s,0):\,s\in\R\}$ of $H^{2m+1}$.

The Heisenberg group $H^{2m+1}$ admits a unitary representation,
denoted by $\vartheta$,
 on the \emph{configuration space} $L^2(\R^m)$ by the formula
\begin{equation}\label{eqn:theta}
    \vartheta(g)\varphi(x)=
e^{2\pi i (s+\langle x,\alpha\rangle-\frac12\langle a,\alpha\rangle)}\varphi(x-a),\quad g=(s,a,\alpha).
\end{equation}

This
representation,
referred to as the \emph{Schr\"{o}dinger representation}, is
irreducible and unitary \cite{vN31}.
The symplectic group, or more precisely its double covering,
also acts on the same Hilbert space
 $L^2(\R^m)$.


In order to track the effect of $\operatorname{Aut}(H^{2m+1})$, we recall briefly its construction. The group
$Sp(m,\R)$ acts by automorphisms of $H^{2m+1}$ preserving the center $Z$ pointwise. Composing
$\vartheta$ with such automorphisms $T\in Sp(m,\R)$ one gets a new
representation $\vartheta\circ T$ of $H^{2m+1}$ on $L^2(\R^m)$. Notice that these representations
have the same central character, namely $\vartheta
\circ T(s,0,0)=e^{2\pi i s}\operatorname{id}=\vartheta(s,0,0)$. According to the Stone--von
Neumann theorem (see Fact \ref{fact:SvN} below) the representations
$\vartheta$ and $\vartheta\circ T$ are equivalent as irreducible unitary representations of $H^{2m+1}$.
 Thus, there
exists a unitary operator $\mathrm{Met}(T)$ acting on $L^2(\R^m)$ in such a
way that
\begin{equation}\label{eqn:32}
\left(\vartheta\circ T\right)(g)=\mathrm{Met}(T)\vartheta(g)\mathrm{Met}(T)^{-1},\quad g\in H^{2m+1}.
\end{equation}
Because $\vartheta$ is irreducible, Met is defined up to a scalar and gives rise to
a projective unitary representation of $Sp(m,\R)$.
It is known that
 this scalar factor may be chosen in one and only one way, up
to a sign, so that $\mathrm{Met}$ becomes a double-valued representation of
$Sp(m,\R)$. The resulting unitary representation of the metaplectic
group, that we keep denoting $\mathrm{Met}$, is referred to as the
\emph{Segal--Shale--Weil representation} and it is a lowest weight module with respect
to a fixed Borel subalgebra.
Notice that choosing
the opposite sign of the scalar factor in the definition of Met one gets a highest weight module which is isomorphic to
the contragredient representation $\mathrm{Met}^\vee$.

The unitary representation Met splits into
 two irreducible and inequivalent subrepresentations
 ${\mathrm{Met}}_0$ and ${\mathrm{Met}}_{1}$ according to the
 decomposition of the Hilbert space
 $L^2({\mathbb{R}}^m)=L^2({\mathbb{R}}^m)_{\mathrm{even}}\oplus
 L^2({\mathbb{R}}^m)_{\mathrm{odd}}$.\vskip10pt

The \emph{Weyl quantization}, or the\emph{ Weyl
operator calculus}, is a way to associate to a function
${\mathfrak S}(x,\,\xi)$
the operator \,${\mathrm{Op}}({\mathfrak S})$ on $L^2(\mathbb{R}^m)$
defined by the equation
\begin{equation}\label{eqn:33}
({\mathrm{Op}}({\mathfrak S})\,u)(x)=\int_{\R^m\times\R^m} {\mathfrak S}\left(\frac{x+y}{2},\,\eta\right)\,
e^{2\pi i\langle x-y,\,\eta\rangle}\,u(y)\,\,dy\,d\eta\,.
\end{equation}
Such a linear operator sets up an isometry
\begin{equation}
\label{eqn:Op}
     {\mathrm{Op}}: L^2({\mathbb{R}}^{2m})
\overset\sim\longrightarrow
\operatorname{ HS}(L^2({\mathbb{R}}^m), L^2({\mathbb{R}}^m)).
\end{equation}
 from the \emph{phase space} \,$L^2(\R^m\times\R^m)$\, onto the Hilbert space consisting of all Hilbert--Schmidt
operators on the configuration space \,$L^2(\R^m)$\,.
Introducing the {\em
symplectic\/} Fourier transformation \,${\mathcal
F}_{\mathrm{symp}}$ by:
\begin{equation}\label{eqn:Fsymp}
({\mathcal F}_{\mathrm{symp}}\,{\mathfrak S})(X):=\int_{\R^m\times \R^m} {\mathfrak S}(Y)\,e^{-2i\pi\omega(X,Y)}\,dY\,,
\end{equation}
one may give another, fully equivalent, definition of the Weyl
operator by means of the equation
\begin{equation}\label{expdef}
{\mathrm{Op}}({\mathfrak S})=\int_{\R^{2m}} ({\mathcal F}_{\mathrm{symp}}\,{\mathfrak S})(Y)\,
\vartheta(0,Y)\,dY,
\end{equation}
where the right-hand side is a Bochner operator-valued integral.

 The Heisenberg group $H^{2m+1}$
acts on $\R^{2m}\simeq H^{2m+1}/Z,$
by
$$
\R^{2m} \to \R^{2m}, \  X\mapsto X+A
\quad\text{for $g = (s,A)$},
$$
 and consequently it acts on the phase space $L^2(\R^{2m})$
by left translations.
The symplectic group $Sp(m,\R)$ also acts on the same Hilbert space
$L^2(\R^{2m})$ by left translations. (This representation is reducible.
See Section \ref{sec:75} for
its irreducible decomposition.)
In fact, both representations come from an action on $L^2(\R^{2m})$ of
the semidirect product group $G^J:=Sp(m,\R)\ltimes H^{2m+1}$ which is referred to as the \emph{Jacobi group}.\vskip10pt

Let us recall some classical facts in a way that we shall use them in the sequel:
\begin{fact}
\label{fact:1} ~
\begin{enumerate}
\item[1)] The representations $\vartheta$ and $Met$ form a  unitary representation of the double covering
$Mp(m,\R)\ltimes H^{2m+1}$  of $G^J$ on the configuration space $L^2(\R^m)$. This action induces a representation
of the Jacobi group $G^J$ on the
 Hilbert space of Hilbert--Schmidt operators $\operatorname{ HS}(L^2(\R^m),L^2(\R^m))$
 by conjugations.\vskip5pt

\item [2)] The Weyl quantization map Op intertwines the action of $G^J$ on $L^2(\R^{2m})$ with the representation $\mathrm{Met}\ltimes\vartheta$
on the Hilbert space
$\operatorname{ HS}(L^2(\R^m),L^2(\R^m))$ defined in 2). Namely,
 \begin{equation}\label{eqn:thetacov}
 \vartheta(g)\,{\mathrm{Op}}({\mathfrak
 S})\,\vartheta(g^{-1})={\mathrm{Op}}({\mathfrak
 S}\,\circ\,g^{-1}),\quad g\in H^{2m+1}.
 \end{equation}
\begin{equation}\label{eqn:1.12}
{\mathrm{Met}}(g)\,{\mathrm{Op}}({\mathfrak S})\,{\mathrm{Met}}^{-1}(g)={\mathrm{Op}}({\mathfrak S}\,\circ\,g^{-1}),\quad g\in Sp(m,\R).
\end{equation}
\item[3)] Any unitary operator satisfying (\ref{eqn:thetacov}) and (\ref{eqn:1.12}) is a scalar multiple of the Weyl quantization map Op.
\end{enumerate}
\end{fact}
\begin{proof}
Most of these statements may be found in the literature (\emph{e.g.} \cite[Chapter 2]{Foll} for the second statement),
but we give a brief explanation of some of them for
the convenience of the reader.  Namely,
the first statement follows from (\ref{eqn:32}). Consequently, the semi-direct product $Mp(m,\R)\ltimes H^{2m+1}$ also acts by
conjugations on the space $\operatorname{ HS}(L^2(\R^m),L^2(\R^m))$, and this action is well
defined for the Jacobi group $G^J=Sp(m,\R)\ltimes H^{2m+1}$ because the
kernel of the metaplectic cover
$Mp(m,\R) \to Sp(m,\R)$ acts trivially on
$\operatorname{ HS}(L^2(\R^m), L^2(\R^m))$.

The third statement follows from the fact that $L^2(\R^{2m})$ is already irreducible by the codimension one subgroup $Sp(m,\R)\ltimes\R^{2m}$
 of $G^J$. Indeed,
any translation--invariant closed subspace of $L^2(\R^{2m})$ is a Wiener space, \emph{i.e.} the pre-image by the Fourier transform of $L^2(E)$ for some measurable set $E$ in $\R^{2m}$.
 On the other hand, the symplectic group acts ergodically on $\R^{2m}$,
in the sense that the only
$Sp(m,\R)$--invariant measurable subsets of $\R^{2m}$ are either null or conull with respect to the Lebesgue measure. Hence, the whole group $Sp(m,\R)\ltimes \R^{2m+1}$ acts irreducibly on $L^2(\R^{2m})$.
\end{proof}

Now we consider the `twist' of the metaplectic representation by automorphisms of the Heisenberg group.

The group of automorphisms of the Heisenberg group
$H^{2m+1}$, to be denoted by Aut$(H^{2m+1})$, is generated by
\begin{itemize}
  \item[-] symplectic maps : $(s,A)\mapsto (s, T(A))$, where $T\in Sp(m,\R)$;
  \item[-] inner automorphisms $(s,A)\mapsto I_{(t,B)}(s,A):=(t,B)(s,A)(t,B)^{-1}$ $=(s-\omega(A,B),A)$, where $(t,B)\in H^{2m+1}$;
  \item[-] dilations  $(s,A)\mapsto d(r)(s,A):=(r^2s,rA)$, where $r>0$;
  \item[-] inversion: $(s,A)\mapsto i(s,A):=(-s,\alpha,a),$ where $A=(a,\alpha)$.
\end{itemize}

In the sequel we shall pay a particular attention to the \emph{rescaling map} $\tau_\rho$ which
is defined for every $\rho\neq0$ by
\begin{equation}\label{eqn:rescale}
{\tau_\rho}: H^{2m+1}\to H^{2m+1},\quad(s,a,\alpha)\mapsto
\left(\frac\rho4\, s,a,\frac\rho4\, \alpha\right).
\end{equation}

Here we have adopted the parametrization of $\tau_\rho$ in a way that
it fits well
into Lemma \ref{lem:4.2}. We note that $(\tau_{-4})^2 = \operatorname{id}$ and $\tau_4={\mathrm{id}}$.

The whole group $\mathrm{Aut}(H^{2m+1})$
of automorphisms is generated by $G^J$ and $\{\tau_\rho:\rho\in\R^\times\}$. We denote by Aut$(H^{2m+1})_o$ the identity component of Aut$(H^{2m+1})$.
Then we have
$$
\mathrm{Aut}(H^{2m+1})=\{1,\tau_{-4}\}\cdot\mathrm{Aut}(H^{2m+1})_o.
$$
\noindent

\vskip10pt

For any given automorphism $\tau\in\mathrm{Aut}(H^{2m+1})$,
we denote by $\overline\tau$
the induced linear operator on
$H^{2m+1}/Z\simeq\R^{2m}$ and by
$\pi(\tau)$ its pull-back
$\pi(\tau)f:= f\circ(\overline\tau)^{-1}$. We notice
 that $\pi(\tau)$ is a unitary operator on $L^2(\R^{2m})$ if $\tau\in G^J$.

Further,
 we define the $\tau$-\emph{twist}
 ${\mathrm{Op}}_\tau$  of the Weyl quantization map ${\mathrm{Op}}$ by
\begin{equation}\label{eqn:Optau}
 {\mathrm{Op}}_\tau:={\mathrm{Op}}\circ\pi(\tau).
\end{equation}
In particular, it follows from \eqref{eqn:33} and \eqref{eqn:rescale} that
\begin{equation}\label{eqn:Oprho}
    ({\mathrm{Op}}_{\tau_\rho}({\mathfrak S})\,u)(x)=\int_{\R^m\times\R^m} {\mathfrak S}\left(\frac{x+y}{2},\,\frac4\rho\xi\right)\,
e^{2\pi i\langle x-y,\,\xi\rangle}\,u(y)\,\,dy\,d\eta\,.
\end{equation}

Similarly, we define the $\tau$-twist $\vartheta_\tau$ of the Schr\"odinger representation $\vartheta$ by
\begin{equation}\label{eqn:thetatau}
\vartheta_\tau:=\vartheta\circ\tau^{-1}.
\end{equation}
Finally, we define the $\tau$-twist $\mathrm{Met}_\tau$ of the Segal--Shale--Weil representation Met.
For this,
we begin with the identity component
 $\mathrm{Aut}(H^{2m+1})_o$.
We set
\begin{alignat}{1}\label{eqn:mettau}
\mathrm{Met}_\tau:={} &A^{-1}\circ \mathrm{Met}\circ A, \qquad\mathrm{where}\\
 &A=\left\{\begin{array}{ll}
  \mathrm{Met}(\tau),&\mathrm{for}\, \tau\in Sp(m,\R),\\
     \vartheta(\tau),& \mathrm{for}\, \tau\in H^{2m+1},
     \\
      \mathrm{Id}, & \mathrm{for}\, \tau=d(r).
       \end{array}\right.\nonumber
\end{alignat}

It follows from Fact \ref{fact:1} 1) that $\mathrm{Met}_\tau$ is
well-defined for $\tau\in\mathrm{Aut}(H^{2m+1})_o$.
For the connected component containing $\tau_{-4}$,
we set
\begin{equation}\label{eqn:Met2}
\mathrm{Met}_\tau:=\left(\mathrm{Met}_{\tau'}\right)^\vee
\end{equation}
for $\tau = \tau_{-4} \tau'$, $\tau' \in \mathrm{Aut}(H^{2m+1})_o$.

Thereby, Met${}_\tau$ is a unitary representation of $Mp(m,\R)$ on $L^2(\R^m)$ characterized for
every $T\in Sp(m,\R)$ by
$$
Met_\tau(T)\vartheta_\tau(g)Met_\tau(T)^{-1}=\vartheta_\tau( T(g)).
$$


Hence, the group $\mathrm{Aut}(H^{2m+1})$ acts on $L^2(\R^{2m})$ in such a way that the following proposition holds.

\begin{prop}
\label{fact:irred} ~
\begin{enumerate}
\item[1)]
 The $\tau$-twisted Weyl
calculus is covariant with
 respect to the Jacobi group:
 \begin{equation}\label{eqn:thetarhocov}
 \vartheta_\tau(g)\,{\mathrm{Op}}_\tau({\mathfrak
 S})\,\vartheta_\tau(g^{-1})={\mathrm{Op}}_\tau({\mathfrak
 S}\,\circ\,g^{-1}),\quad g\in H^{2m+1},
 \end{equation}
\begin{equation}\label{eqn:metrho}
{\mathrm{Met}}_\tau(g)\,{\mathrm{Op}}_\tau({\mathfrak S})\,{\mathrm{Met}}^{-1}_\tau(g)=
{\mathrm{Op}}_\tau({\mathfrak S}\,\circ\,g^{-1}), \quad g\in Sp(m,\R).
\end{equation}
\item[2)] For any $\tau\in\mathrm{Aut}(H^{2m+1})$ the representation $Met_\tau$ is equivalent either to Met or to its
contragredient Met${}^\vee$.
\end{enumerate}
\end{prop}

%

The special case of the $\tau$-twist,
namely,
the $\tau$-twist associated with the rescaling map $\tau_\rho$ (\ref{eqn:rescale}) deserves our attention for at least the following two reasons.
 First, the parameter $\frac\rho4$ has a
concrete physical meaning - this is the inverse of the Planck constant $h$ (see \cite[Theorem 4.57]{Foll}, where a slightly different notation was
used. Namely, the Schr\"odinger
representations that we denote by
$\vartheta_{\tau_\rho}$ correspond therein to $\rho_{h}$ with $h=\frac4\rho$).
 Secondly, dilations
do not preserve the center $Z$ of the Heisenberg while the symplectic automorphisms of $H^{2m+1}$ do. More precisely, the whole Jacobi
group $G^J$ fixes $Z$ pointwise.
The last observation together with the Stone -- von Neumann theorem (see below) shows that the action of ${\mathrm{Aut}}(H^{2m+1})/G^J\simeq\{\tau_\rho:\rho\in\R^\times\}( \simeq \R^\times)$
is sufficient in order to obtain all infinite dimensional irreducible unitary representations of the Heisenberg group.

We set
\begin{equation}\label{eqn:thetarho}
\vartheta_\rho:=\vartheta_{\tau_\rho},
\end{equation}
to which we refer as the \emph{Schr\"{o}dinger representations}
with central character $\rho$.

\begin{fact}[{Stone--von Neumann Theorem, \cite{H80,vN31}}]
\label{fact:SvN}
The
representations $\vartheta_\rho$
constitute a
family of irreducible pairwise inequivalent unitary
representations with real parameter $\rho$. Any infinite dimensional irreducible unitary
representation of $H^{2m+1}$ is uniquely determined by its central character and thus
equivalent to one of the $\vartheta_\rho$'s.
\end{fact}

%
%


\vskip10pt

%
%
%

To end this section,
we give yet another algebraic property of the Weyl operator calculus.
We shall see in Lemma \ref{lem:1.8} that the irreducible decomposition  of 
$\pi_{i\lambda,\delta}^{GL(2n,\R)}$, when restricted to a maximal parabolic subgroup of
$Sp(n,\R)$, is based on an
involution of the phase space
coming from the parity preserving involution on the configuration
space. 

Consider on $L^2(\R^m)$ an involution defined by $\check{u}(x):=u(-x)$ and
induce through the map ${\mathrm{Op}_{\tau_\rho}}:L^2(\R^{2m})\to \operatorname{ HS}
(L^2(\R^m),L^2(\R^m))$ two involutions on $L^2(\R^{2m})$, denoted by
$\mathfrak S\mapsto{}^{\dag_\rho}\mathfrak S$ and $\mathfrak S\mapsto\mathfrak
S^{\dag_\rho}$, by the following identities:
\begin{eqnarray}\label{eqn:dagger}
  {\mathrm{Op}_{\tau_\rho}} ({}^{\dag_\rho}\mathfrak S)(u)&=& {\mathrm{Op}_{\tau_\rho}}(\mathfrak S)(\check{u}), \\
  {\mathrm{Op}_{\tau_\rho}}( \mathfrak S^{\dag_\rho})(u)&=&({\mathrm{Op}_{\tau_\rho}}(\mathfrak S)(u))\check{}.\label{eqn:dagger2}
\end{eqnarray}
Then 
${}^{\dag_\rho}\mathfrak S$
and
$\mathfrak S^{\dag_\rho}$
are characterized by their partial Fourier transforms defined by
$$
(\mathcal F_\xi\mathfrak S)(x,\eta)
:=\int_{\R^m}\mathfrak S(x,\xi) e^{-2\pi i \langle \xi,\eta\rangle}d\xi
\quad\text{for $\mathfrak S\in L^2(\R^{2m})$}.
$$
\begin{lem}\label{lem:invol}
\begin{eqnarray*}
  \left(\mathcal F_\xi{}^{\dag_\rho}\mathfrak S\right)(x,\eta) &=& \left(\mathcal F_\xi\mathfrak S\right)\left(-\frac2\rho\eta,-\frac\rho2x\right),  \\
    \left(\mathcal F_\xi\mathfrak S^{\dag_\rho}\right)(x,\eta) &=& \left(\mathcal F_\xi\mathfrak S\right)\left(\frac2\rho\eta,\frac\rho2x\right).
\end{eqnarray*}
\end{lem}
\begin{proof}
By \eqref{eqn:Oprho} the first equality \eqref{eqn:dagger} amounts to
\begin{eqnarray*}
    && \int_{\R^m\times\R^m}{}^{\dag_\rho}\mathfrak S\left(\frac{x+y}2,\frac4\rho\xi\right)e^{2i\pi\langle x-y,\xi\rangle}u(y)dyd\xi \\
  &=&  \int_{\R^m\times\R^m}\mathfrak S\left(\frac{x+y}2,\frac4\rho\xi\right)e^{2i\pi\langle x-y,\xi\rangle}u(-y)dyd\xi.
\end{eqnarray*}

The right-hand side equals
$$
  \left(\frac{|\rho|}4\right)^n\int_{\R^m\times\R^m}\mathfrak S\left(\frac{x-y}2,\xi\right)e^{2i\pi\langle \frac\rho4(x+y),\xi\rangle}u(y)dyd\xi.
$$
This equality holds for all $u\in L^2(\R^m)$, and therefore,
\begin{equation*}
    \int_{\R^m}{}^{\dag_\rho}\mathfrak S\left(\frac{x+y}2,\frac4\rho\xi\right)e^{2i\pi\langle x-y,\xi\rangle}d\xi =
     \left(\frac{|\rho|}4\right)^n\int_{\R^m}\mathfrak S\left(\frac{x-y}2,\xi\right)e^{2i\pi\langle \frac\rho4(x+y),\xi\rangle}d\xi.
\end{equation*}
Namely,
$$
\left(\mathcal F_\xi{}^{\dag_\rho}\mathfrak S\right)\left(\frac{x+y}2,\frac\rho4(y-x)\right) = \left(\mathcal F_\xi\mathfrak S\right)\left(\frac{x-y}2,-\frac\rho4(x+y)\right).
$$
Thus the first statement follows and the second may be proved in the
same way.
\end{proof}

\section{Restriction of $\pi_{i\lambda,\delta}$ to a maximal parabolic subgroup}\label{sec:4}

Let $n=m+1$.
Consider the space of homogeneous functions
\begin{equation}\label{eqn:homoRN}
V_{\mu,\delta}^\infty
:=\{f\in C^\infty(\R^{2n}\setminus\{0\}):f(r\,\cdot)=(\operatorname{sgn}r)^\delta|r|^{-n-\mu}f(\cdot), r\in\R^\times\},
\end{equation}
for $\delta=0,1$ and $\mu\in\C$. It may be seen as the space of even or odd smooth functions
on the unit sphere $S^{2n-1}$ according to $\delta=0$ or $1$, since
homogeneous functions are determined by their restriction to $S^{2n-1}$.
Let $V_{\mu,\delta}$ denote its completion
with respect to the $L^2$-norm over $S^{2n-1}$.
Likewise, by restricting to the hyperplane defined by the first coordinate to be $1$, we can
identify the space $V_{\mu,\delta}$ with the Hilbert space $L^2(\R^{2n-1})$ up to a scalar multiple on the inner product.

The normalized degenerate principal series representations
$\pi_{\mu,\delta}^{GL(2n,\R)}$ induced from the character $\chi_{\mu,\delta}$ of
a maximal parabolic subgroup $P_{2n}$ of
$GL(2n,\R)$ corresponding to the partition $2n=1+(2n-1)$ may be realized on these functional spaces.
The realization of the same representation on $V_{\mu,\delta}$ will
be referred to as the $K$-\emph{picture}, and on $L^2(\R^{2n-1})$ as the $N$-\emph{picture}.

In addition to these standard models of $\pi_{\mu,\delta}^{GL(2n,\R)}$, we shall use another model $L^2(\R,{{\operatorname{ HS}}}(L^2(\R^{m}),L^2(\R^{m})))$,
 which we call the \emph{operator calculus model}. It gives a strong machinery for investigating the restriction to the maximal parabolic subgroup of $Sp(n,\R)$ (see (\ref{eqn:PSpn}) below).
 
 Let us denote by
$$\displaystyle
\mathcal{F}_t(f)(\rho,X)=\int_\R f(t,X)\,e^{-2i\pi t\rho}\,dt,
$$
the partial Fourier transform of $f(t,X)\in L^2(\R^{1+2m})$ with respect to the first
variable. Applying the direct integral of the operators $\operatorname{Op}_{\tau_\rho}$ and using \eqref{eqn:FLk},
we obtain the unitary isomorphisms
\begin{eqnarray}\label{eqn:LMFR}
  && V_{\mu,\delta} \simeq
   L^2(\mathbb{R}^{1+2m}) \simeq
L^2(\mathbb{R},L^2(\mathbb{R}^{2m}))
\underset{\mathcal{F}_t}{\simeq}
L^2(\mathbb{R},L^2(\mathbb{R}^{2m}))\\\nonumber &&\underset{\int
{\mathrm{Op}}_{\tau_\rho}d\rho}{\stackrel{\sim}{\longrightarrow}} L^2(\mathbb R,
{\mathrm{HS}}(L^2(\mathbb R^m),L^2(\mathbb R^m))).
\end{eqnarray}

According to situations we shall use following geometric models for the induced representations:

\begin{figure}[H]
$$
\begin{matrix}
&\text{standard model}
\\[1ex]
&\hidewidth
\kern8.5em
 \framebox[34.7em][l]{$\begin{array}{ll}
          V_{\mu,\delta} = L^2(S^{2n-1})_\delta
         &\quad\text{$K$-picture}
       \\
          \kern-4em\rightsetnw{\text{restrict}}
       \\
          \kern4em
          V_{\mu,\delta}^\infty = \{f \in C^\infty(\mathbb{R}^{2n}\setminus\{0\}):
                             f(r X)
                             = |r|^{-\mu-n}
                             (\operatorname{sgn}r)^\delta f(X),
                             r \in \mathbb{R}^\times \}
          \hidewidth
       \\
          \kern-4em\rightsetsw{\text{restrict}}
       \\
          L^2(H^{2m+1}) = L^2(\mathbb{R},L^2(\mathbb{R}^{2m}))
         &\quad\text{$N$-picture}
       \end{array}
       $}\hidewidth
\\
\\[-1ex]
&\rightsetd{\mathcal{F}_t}
\\[2ex]
&\hidewidth
 \fbox{$L^2(\mathbb{R},L^2(\mathbb{R}^{2m})) \simeq L^2(\mathbb{R}^{2m+1})$}
 \hidewidth
\\[2ex]
&\leftsetsw{\mathcal{F}_\xi}
 \kern10em
 \rightsetse{\int_{\mathbb{R}} \mathrm{Op}_{\tau_\rho}d\rho}
\\[2ex]
\fbox{\parbox{8em}{\begin{center}$\mathcal U_{\mu,\delta}=L^2(\mathbb{R}^{2m+1})$
\\
\text{(see Section \ref{sec:6})}\end{center}}}
&&
 \fbox{\parbox{13em}{\begin{center}
$L^2(\mathbb{R},\mathrm{HS}(L^2(\mathbb{R}^m),L^2(\mathbb{R}^m)))$
\\
\text{(see Section \ref{sec:5})}\end{center}}}
\\
\mbox{non-standard model}&&\mbox{operator calculus model}
\end{matrix}
$$
\caption{}
\label{fig:3models}
\end{figure}

The group $G_1=Sp(n,\R)(=Sp(m+1,\R))$ acts by linear symplectomorphisms
on $\R^{2n}$ and thus it also acts on the real projective space
$\mathbb P^{2m+1}\R$ . Fix a point in $\mathbb P^{2m+1}\R$ and denote by $P$
its stabilizer in $G_1$. This is a maximal parabolic subgroup of $G_1$ with Langlands decomposition
\begin{equation}\label{eqn:PSpn}
P  = MA\overline{N}
\simeq (\R^\times \cdot Sp(m,\R)) \ltimes H^{2m+1}.
\end{equation}

Let $\mathfrak g_1=\mathfrak n+\mathfrak m+\mathfrak a+\overline{\mathfrak n}$ be the Gelfand--Naimark decomposition for the Lie algebra $\mathfrak g_1=Lie (G_1)$.

We identify the standard Heisenberg Lie group $H^{2m+1}$ with the  subgroup $N=\exp \mathfrak n$ through the following Lie groups isomorphism:
\begin{equation}\label{eqn:hn}
(s,x,\xi)\mapsto\left(
                  \begin{array}{cccc}
                    1 & 0 & 0 & 0 \\
                    x & I_m & 0 & 0 \\
                    2s & {}^t\xi & 1 & -^tx \\
                    \xi & 0 & 0 & I_m\\
                  \end{array}
                \right).
\end{equation}
Thus,
in the coordinates
$(t,x,\xi) \in H^{1+2m}$,
the restriction map
$V_{\mu,\delta}^\infty \to L^2(H^{2m+1})$
is given by
\begin{equation}\label{eqn:Npic}
f \mapsto f(1,2t,x,\xi).
\end{equation}

The action of $G_1$  on $\mathbb{P}^{2n-1}\R$ is transitive, and all
such isotropy subgroups are conjugate to each other. Therefore, we may assume
that $P=Sp(n,\R)\cap P_{2n}$.
Then, the natural inclusion
$Sp(n,\mathbb{R}) \subset GL(2n,\mathbb{R})$
induces the following isomorphisms
\[
Sp(n,\mathbb{R})/P \overset{\sim}{\to}
GL(2n,\mathbb{R})/P_{2n} \simeq \mathbb{P}^{2n-1}\mathbb{R}.
\]
Hence, the (normalized) induced representation $\pi_{\mu,\delta}\equiv
\pi_{\mu,\delta}^{Sp(n,\R)}:=
{\mathrm
{Ind}}_P^{Sp(n,\R)}\chi_{\mu,\delta}$ can (\emph{cf}. Section \ref{sec:prf})
also be realized on the Hilbert space $V_{\mu,\delta}$.
Therefore, $\pi_{\mu,\delta}$ is
 equivalent to the restriction of
$\pi_{\mu,\delta}^{GL(2n,\R)}$ with respect to $Sp(n,\R)$. Notice that $\pi_{\mu,\delta}$
is unitary for $\mu=i\lambda,\,\lambda\in\R$.

It is noteworthy  that the unipotent radical $\overline N$ of $P$ is the Heisenberg
group $H^{2n-1}$ which is not abelian if $n\ge2$,
although the unipotent radical of $P_{2n}$ clearly is.
Notice also that the automorphism group Aut($H^{2n-1}$) contains $P/\{\pm1\}$
as a subgroup of index 2.

Denote by $M_o\simeq Sp(m,\R)$ the identity component of
$M\simeq O(1) \times Sp(m,\mathbb{R})$.
The subgroup
 $M_o\ltimes N$ is isomorphic to the
Jacobi group $G^J$ introduced in Section \ref{sec:3}.


We have then the following inclusive relations for subgroups of
symplectomorphisms:
\[
\begin{matrix}
G_1 & \supset & MAN & \supset & G^J=M_oN & \supset &N.
\\[-1ex]
\text{\footnotesize Symplectic group} &&&& \text{\footnotesize Jacobi group} &&
\text{\footnotesize Heisenberg group}
\end{matrix}
\]

 Our
strategy of analyzing the representations
$\pi_{i\lambda,\delta}$ of $G_1$ (see Theorem \ref{thm:twosum}) will be based on their restrictions to these subgroups (see Lemmas \ref{lem:MNdeco} and \ref{lem:1.8}).

We recall from \eqref{eqn:thetarho} that $\vartheta_\rho$ is the
Schr\"odinger representation of the Heisenberg group $H^{2m+1}$
with central character $\rho$.
While the abstract Plancherel formula for the group 
$N \simeq H^{2m+1}$:
$$
L^2(N)=\int_\R \vartheta_\rho\otimes \vartheta_\rho^\vee d\rho,
$$
underlines the decomposition with respect to left and right regular actions of the group $N$,
we shall consider the decomposition of this space with respect to the
 restriction of the principal series representation $\pi_{i\lambda,\delta}$ to the Jacobi group
 $G^J=Sp(m,\R)\ltimes H^{2m+1}$ (see Lemma \ref{lem:MNdeco}).

Let us examine how the restriction $\pi_{i\lambda,\delta}|_{G^J}$ defined on the Hilbert space $V_{i\lambda,\delta}$
 on the left-hand side of \eqref{eqn:LMFR}
is transferred to $L^2(\mathbb{R},L^2(\mathbb{R}^{2m}))$ via the
partial Fourier transform $\mathcal{F}_t$.

The restriction $\pi_{i\lambda,\delta}|_N$ coincides with the left regular representation of $N$ on
 $L^2(\R^{1+2m})$ given by
\begin{align}\label{eqn:piN}
   \pi_{i\lambda,\delta}(g)f(t,X)
      &= f(t-s-\frac12\omega(A,X),
      X-A)\\
      &= f(t-s+\frac12(\langle \xi,a\rangle - \langle x,\alpha\rangle),
      x - a,\xi-\alpha),
\nonumber
\end{align}
for $f(t,X)\in L^2(\R^{1+2m})$ and $g=(s,A)\equiv(s,a,\alpha)\in\H^{2m+1}$.

Taking the partial Fourier transform $\mathcal{F}_t$ of \eqref{eqn:piN},
 we get
\begin{equation}\label{eqn:FpiT}
({\mathcal F}_t\,(\pi_{i\lambda,\delta}(g)\,f))(\rho,x,\,\xi)=e^{-2\pi i\rho\,(s-\frac12(\langle \xi,a\rangle- \langle x,\alpha\rangle)
)}\,({\mathcal F}_tf)(\rho,x-a,\,\xi-\alpha).
\end{equation}

Now, for each $\rho \in \mathbb{R}$,
 we define a representation $\varpi_\rho$ of
$N$ on $L^2(\mathbb{R}^{2m})$ by
\begin{equation}\label{eqn:varpi}
  \varpi_\rho(g) h(x,\xi)
 := e^{-2\pi i\rho\,(s- \frac12(\langle \xi,a\rangle- \langle x,\alpha\rangle))}\,
  h(x - a, \xi - \alpha),
\end{equation}
for $g=(s,a,\alpha) \in N$ and $h \in L^2(\mathbb{R}^{2m})$. Then,
$\varpi_\rho$ is a unitary representation of $N$ for any $\rho$, and
the formula \eqref{eqn:FpiT} may be written as:
\begin{equation}\label{eqn:Fpig}
(\mathcal{F}_t \pi_{i\lambda,\delta} (g)f)(\rho,x,\xi)
= \varpi_\rho (g)(\mathcal{F}_t f)(\rho,x,\xi),
\end{equation}
for $g \in N$. Here, we let $\varpi_\rho(g)$ act on $\mathcal{F}_t
f$ seen as a function of $(x,\xi)$.

 For each $\rho \in
\mathbb{R}$, we can extend the representation $\varpi_\rho$ of $N$
to a unitary
 representation of the Jacobi group $G^J$ by
letting $M_o$ act on $L^2(\mathbb{R}^{2m})$ by
\begin{equation*}
\varpi_\rho(g)h(x,\xi)= h(y,\eta),\quad {\mathrm{with}}\,
(y,\eta)=g^{-1}(x,\xi), \,g\in M_o\simeq Sp(m,\R).
\end{equation*}
Then, clearly the identity \eqref{eqn:Fpig} holds also for $g \in
M_o$. Thus, we have proved the following decomposition formula:
\begin{lem}\label{lem:MNdeco}
For any $(\lambda,\delta) \in \mathbb{R} \times
\mathbb{Z}/2\mathbb{Z}$,
the restriction of $\pi_{i\lambda,\delta}$ to the Jacobi group is
unitarily equivalent to the direct integral of unitary representations
$\varpi_\rho$ via $\mathcal{F}_t$
(see \eqref{eqn:LMFR}):
\begin{equation}\label{eqn:MNdeco}
\pi_{i\lambda,\delta} |_{G^J}
\underset{\mathcal{F}_t}{\simeq}
\int_{\mathbb{R}}^\oplus \varpi_\rho d\rho.
\end{equation}
\end{lem}
\vskip10pt
Next we establish the link between the representations ($\varpi_\rho,L^2(\R^{2m}))$ and $(\vartheta_\rho,L^2(\R^m))$ of the Heisenberg group $N\simeq H^{2m+1}$ .
For this we note that the representation $\varpi_\rho$ brings us to the changeover of one parameter
families of automorphisms of $H^{2m+1}$, from $\{\tau_\rho:\rho\in\R^\times\}$ to 
$\{\psi_\rho:\rho\in\R^\times\}$ which defined by
\begin{equation}\label{eqn:psirho}
\psi_\rho(s,a,\alpha):=\left(\frac1\rho s,\frac12 a,\frac2\rho\,\alpha\right).
\end{equation}
 
 Then we state the following covariance relation given by $ {\mathrm{Op}}_{\tau_\rho}$:

\begin{lem}\label{lem:4.2}
 For every $g\in H^{2m+1}$ the following identity in $\mathrm{End}(L^2(\R^m))$ holds for any $\mathfrak S\in L^2(\R^{2m})$ :
\begin{equation}\label{eqn:Op-phi-rho}
    {\mathrm{Op}}_{\tau_\rho}(\varpi_\rho(g)\mathfrak S)=
    {\mathrm{Op}}_{\tau_\rho}(\mathfrak S)\circ \vartheta_{\psi_\rho}(g^{-1}).
\end{equation}
\end{lem}
\begin{proof}
Let $g=(s,a,\alpha)\in H^{2m+1}$ and take an arbitrary function $u\in L^2(\R^m)$.
Using the integral formula \eqref{eqn:Oprho} for ${\mathrm{Op}}_{\tau_\rho}$, we get
\begin{eqnarray*}
 &&{\mathrm{Op}}_{\tau_\rho}(\varpi_\rho(g)\mathfrak S)u(x) \\
   &=& \int_{\R^m\times\R^m}\left( \varpi_\rho(g)\mathfrak S\right)\left(\frac{x+y}2,\frac4\rho\xi\right)e^{2\pi i\langle x-y,\xi\rangle}
   u(y)dyd\xi\nonumber\\
   &=& \int_{\R^m\times\R^m} e^{-2\pi i\rho \left( s-\frac12\left(\langle\frac4\rho\xi,a\rangle-\left\langle\frac{x+y}2,\alpha\right\rangle
    \right)\right)}\mathfrak S\left(\frac{x+y}2-a,\frac4\rho\xi-\alpha\right)e^{2\pi i\langle x-y,\xi\rangle}
   u(y)dyd\xi\nonumber\\
   &=&
   \int_{\R^m\times\R^m}e^{-2\pi i B}
   \mathfrak S\left(\frac{x+y}2,\frac4\rho \xi\right)
      u(y+2a)dyd\xi,
   \end{eqnarray*}
   where
   \begin{eqnarray*}
   B&=&
    \rho s-\frac\rho2\left(\langle\frac4\rho\xi+\alpha,a\rangle-\left\langle\frac{x+y+2a}2,\alpha\right\rangle\right)-
   \left\langle x-y-2a,\xi+\frac\rho4\alpha\right\rangle\\
   &=&\rho s+\frac\rho2\langle a+y,\alpha\rangle-\langle x-y,\xi\rangle.
   \end{eqnarray*}
   In view of the definitions \eqref{eqn:thetatau} anf \eqref{eqn:psirho},
   $$
   \vartheta_{\psi_\rho}(g^{-1})=\vartheta(\psi_\rho^{-1}g^{-1})
=\vartheta(-\rho s,-2a,-\frac\rho2 \alpha).
   $$
   Thus, by the definition \eqref{eqn:theta} of the Schr\"odinger representation $\vartheta$, 
we have
\[
  (\vartheta_{\psi_\rho}(s^{-1})u)(y)
  = e^{-2\pi i(\rho s+\frac{\rho}{2}\langle a+y,\alpha\rangle)}
    u(y+2a).
\]
Hence,
the last integral equals
   \begin{eqnarray*}
   &&\int_{\R^m\times\R^m}
   \mathfrak S\left(\frac{x+y}2,\frac4\rho\xi\right)
   (\vartheta_{\psi_\rho}(g^{-1})u)(y)e^{2\pi i\langle x-y,\xi\rangle}
   dyd\xi\\
   &=&\left({\mathrm{Op}}_{\tau_\rho}(\mathfrak S)\vartheta_{\psi_\rho}(g^{-1})u\right)(x).
\end{eqnarray*}
\end{proof}

Then, it turns out that the decomposition (\ref{eqn:MNdeco}) is not
irreducible, but the following lemma holds:
\begin{lem}\label{lem:1.7}
For any $\rho \in \mathbb{R}^\times$, $\varpi_\rho$ is a unitary
representation of the Jacobi group $G^J$ on $L^2(\mathbb{R}^{2m})$,
which splits into a direct sum $\varpi_\rho^0\oplus\varpi_\rho^1$ of
two pairwise inequivalent unitary irreducible representations.
\end{lem}
\begin{proof}
Consider the rescaling map $\tau_\rho$ introduced by (\ref{eqn:rescale}) and recall that the $\tau_\rho$--twisted
 Weyl quantization map induces a $G^J$ equivariant isomorphism
\begin{equation}
\label{eqn:Op}
     {\mathrm{Op}}_{\tau_\rho}: L^2({\mathbb{R}}^{2m})
\overset\sim\longrightarrow
 \operatorname{ HS}(L^2({\mathbb{R}}^m), L^2({\mathbb{R}}^m))
\end{equation}
intertwining the $\varpi_\rho$ and $\vartheta_{\psi_\rho}$ actions (\ref{eqn:Op-phi-rho}).

 The irreducibility of the Schr\"odinger representation $\vartheta_\rho$ of the group $N$
 (Fact \ref{fact:SvN}) implies therefore
 that any $N$-invariant closed subspace in
$\operatorname{ HS}(L^2(\mathbb{R}^m), L^2(\mathbb{R}^m))$ must be of the form
 $
     \operatorname{ HS}(L^2({\mathbb{R}}^m),U)
$
 for some closed subspace $U\subset L^2({\mathbb{R}}^m)$.

In view of the covariance relation (\ref{eqn:metrho}) of the Weyl
quantization, the subspace $\operatorname{ HS}(L^2({\mathbb{R}}^m),U)$ is
$Sp(m,\mathbb{R})$-invariant
 if and only if $U$ itself is $Mp(m,\mathbb{R})$-invariant (see Proposition
\ref{fact:irred}), and the latter happens only if $U$ is one of
$\{0\}$, $L^2({\mathbb{R}}^m)_{\mathrm{even}}, L^2({\mathbb{R}}^m)_{\mathrm{odd}}$
or $L^2({\mathbb{R}}^m)$. Thus, we have the following
irreducible decomposition of $\varpi_\rho$, seen as a representation
of $G^J$ on $L^2({\mathbb{R}}^{2m})$:
\begin{eqnarray}\label{eqn:decompMN}
    L^2(\mathbb R^{2m})&=&W_+\oplus W_-\nonumber\\
    &\underset{{\mathrm{Op}}_{\tau_\rho}}{\stackrel{\sim}{\longrightarrow}}& 
    \operatorname{HS}(L^2({\mathbb{R}}^m),L^2({\mathbb{R}}^m)_{\mathrm{even}})\oplus
     \operatorname{HS}(L^2({\mathbb{R}}^m),L^2({\mathbb{R}}^m)_{\mathrm{odd}}).
\end{eqnarray}

From Proposition \ref{fact:irred} 2) we deduce that the
corresponding representations,
to be denoted by
 $\varpi_\rho^\delta$, of $G^J$, where
$\delta$ labels the parity, are pairwise inequivalent, i.e.
$\varpi_\rho^\delta=\varpi_{\rho'}^{\delta'}$ if and only if
$\rho=\rho'$ and $\delta=\delta'$ for all $\rho,\rho'\in\mathbb R$
and $\delta,\delta'\in \mathbb Z/2\mathbb Z$.
\end{proof}
The following lemma is straightforward from the definition of the
involution
$\mathfrak S\mapsto\mathfrak S^{\dag_\rho}$
(see \eqref{eqn:dagger}).
\begin{lem}\label{lem:Vpm}
The subspaces $W_+$ and $W_-$ introduced above are the $+1$ and $-1$ eigenspaces
of the involution $\mathfrak S\mapsto\mathfrak S^{\dag_\rho}$,
respectively.
\end{lem}

Eventually, we take the $A$-action into account,
and give the branching law of the (degenerate) principal series
representation $\pi_{i\lambda,\delta}$ of $G_1$ when restricted to the
maximal parabolic subgroup $MAN$.

\begin{lem}\label{lem:1.8} \emph{(Branching law for $G_1\downarrow MAN$).}
For every $(\lambda,\delta)\in\mathbb R\times \mathbb Z/2\mathbb Z$
the space $V_{i\lambda,\delta}$ acted upon by the
representation $\pi_{i\lambda,\delta}|_{MAN}$ splits into the direct
sum of four irreducible representations:
\begin{equation}\label{eqn:AN}
V_{i\lambda,\delta}
\simeq
{\mathcal H}_+^2(W_+)
\oplus
{\mathcal H}_+^2(W_-)
\oplus
{\mathcal H}_-^2(W_+)
\oplus
{\mathcal H}_-^2(W_-).
\end{equation}
\end{lem}
\begin{proof}
We shall prove first that each summand in \eqref{eqn:AN} is already
 irreducible as a representation of
$M_oAN\simeq G^JA$.
Then we see that
it is stable by the group $MAN$ and thus irreducible
because $M$ is generated by $M_o$
and $-I_{2n}$,which acts on $V_{i\lambda,\delta}$ by the scalar $(-1)^\delta$.

In light of the $G^J$-irreducible decomposition \eqref{eqn:MNdeco},
any $G^J$-invariant closed subspace $U$ of $V_{i\lambda,\delta}$ must be of the form
\[
U = \mathcal{F}_t^{-1} (L^2(E_+,W_+))\oplus\mathcal{F}_t^{-1} (L^2(E_-,W_-)),
\]
for some measurable sets $E_\pm$ in $\mathbb{R}$.

Suppose furthermore that $U$ is $A$-invariant. Notice that the group
$A$ acts on $V_{i\lambda,\delta} \simeq
L^2(\mathbb{R}^{2m+1})$
 by
\[
\pi_{i\lambda,\delta}(a) f(t,X)
= a^{-1-m-i\lambda} f(a^{-2} t, a^{-1}X).
\]

In turn, their partial Fourier transforms with respect to the $t\in\mathbb R$
variable are given by
\[
(\mathcal{F}_t \pi_{i\lambda,\delta} (a)f) (\rho,X)
= a^{1-m-i\lambda} (\mathcal{F}_t f) (a^2\rho, a^{-1} X).
\]
Therefore, $\mathcal{F}_t f$ is supported in $E_\pm$ if and only if
$\mathcal{F}_t \pi_{i\lambda,\delta}(a)f$ is supported in $a^{-2}E_\pm$
as a $W_\pm$-valued function on $\mathbb{R}$. In particular, $U$ is
an $A$-invariant subspace if and only if $E_\pm$ is an invariant
measurable set under the dilation $\rho \mapsto a^2\rho$ $(a>0)$,
namely, $E_\pm = \{0\}$, $\mathbb{R}_-$, $\mathbb{R}_+$, or
$\mathbb{R}$ (up to measure zero sets). 

Since $M_oAN\simeq G^JA$, $M_oAN$-invariant
proper closed subspaces must be of the form
$\mathcal{F}_t^{-1}(L^2(\mathbb{R}_\pm,
W_\varepsilon))$ with
$\varepsilon = +$ or $-$.

We recall from Lemma \ref{lem:Hardy} that the Hilbert space $L^2(\R,W_\varepsilon)$ is a sum of $W_\varepsilon$-valued Hardy spaces:
\begin{equation}\label{eqn:LHH}
L^2(\mathbb{R},W_\varepsilon)
= \mathcal{H}_+^2(W_\varepsilon) \oplus \mathcal{H}_-^2(W_\varepsilon)
\underset{\mathcal F_t}{\stackrel{\sim}{\rightarrow}}L^2(\R_+,W_\varepsilon)\oplus L^2(\R_-,W_\varepsilon).
\end{equation}
 Now Lemma
\ref{lem:1.8} has been proved.
\end{proof}

Lemma \ref{lem:1.8} implies that the representation
$\pi_{i\lambda,\delta}$ of $G_1$ has
at most four irreducible subrepresentations.
The precise statement for this will be given in Theorem \ref{thm:twosum} .
\vskip10pt

\section{Restriction of $\pi_{i\lambda,\delta}$ to a maximal compact subgroup}
\label{sec:5}

As the operator calculus model $L^2(\R,\operatorname{HS}(L^2(\R^m),L^2(\R^m)))$
was appropriate  for studying the $P$-structure of $\pi_{i\lambda,\delta}$,
we use complex spherical harmonics for the analysis of  the $K$-structure of these representations.

We retain the convention $n=m+1$. Identifying
the symplectic form $\omega$ on $\R^{2n}$ with the imaginary part of
the Hermitian inner product on $\C^n$ we realize the group of
unitary transformations $K=U(n)$ as a subgroup of $G_1=Sp(n,\R)$.
Then the group $K$ is a maximal compact subgroup of $G_1$.

Analogously to the classical spherical harmonics on $\R^n$, consider
harmonic polynomials on $\C^n$ as follows. For
$\alpha,\beta\in\mathbb N$, let ${\mathcal{H}}^{\alpha,\beta}(\C^n)$
denote the vector space of polynomials $p(z_0,\ldots,z_m ,\bar
z_0,\ldots,\bar z_m)$ on $\C^n$ which
\begin{enumerate}
  \item are homogeneous of degree $\alpha$ in $(z_0,\ldots, z_m)$ and of degree $\beta$ in
$(\bar z_0,\ldots,\bar z_m)$;
  \item belong to the kernel of the differential operator
  $\displaystyle\sum_{i=0}^m\frac{\partial^2}{\partial z_i\partial\bar z_i}$.
\end{enumerate}

Then, ${\mathcal{H}}^{\alpha,\beta}(\mathbb{C}^n)$ is a finite dimensional vector
space.
It is non-zero except for the case where $n=1$ and $\alpha,\beta\ge1$.
The natural action of $K$ on polynomials,
\[
p(z_0,\ldots,z_m ,\bar z_0,\ldots,\bar z_m) \mapsto
p(g^{-1} (z_0,\ldots,z_m),\, \overline{g^{-1} (z_0,\ldots,z_m)})
\quad
(g \in K),
\]
 leaves ${\mathcal{H}}^{\alpha,\beta}(\mathbb{C}^n)$ invariant.
The resulting representations of $K$  on
${\mathcal{H}}^{\alpha,\beta}(\mathbb{C}^n)$, which we denote by the same symbol
${\mathcal{H}}^{\alpha,\beta}(\mathbb{C}^n)$,
are irreducible and
pairwise inequivalent for any such $\alpha, \beta$.
\vskip5pt

The restriction of ${\mathcal{H}}^{\alpha,\beta}(\mathbb{C}^n)$
to the unit sphere
$S^{2m+1} = \{ (z_0,\ldots,z_m) \in \mathbb{C}^n: \sum_{j=0}^m|z_j|^2=1 \}$
is injective
 and gives a complete orthogonal basis of
$L^2(S^{2m+1})$,
and we have a discrete sum decomposition
\begin{equation}\label{eqn:spab}
L^2(S^{2m+1}) \simeq
\sideset{}{^\oplus} \sum_{\alpha,\beta \in \mathbb{N}}
{\mathcal{H}}^{\alpha,\beta} (\mathbb{C}^n) |_{S^{2m+1}}
\quad
(m \ge 1).
\end{equation}
The case $m=0$ collapses to
\[
L^2(S^1) \simeq
\sideset{}{^\oplus} \sum_{\alpha \in \mathbb{N}}
{\mathcal{H}}^{\alpha,0} (\mathbb{C}^1) |_{S^1}
\oplus
\sideset{}{^\oplus} \sum_{\beta \in \mathbb{N}_+}
{\mathcal{H}}^{0,\beta} (\mathbb{C}^1) |_{S^1}.
\]

Fixing a $\mu \in \mathbb{C}$ we may extend functions on
$S^{2m+1}$ to homogeneous functions of degree $-(m+1+\mu)$.
The decomposition \eqref{eqn:spab} gives rise to the branching law
 ({\it{$K$-type formula}})
 with respect to the maximal compact subgroup.
\begin{lem}\label{lem:Kdeco}\emph{(Branching law for $G_1\downarrow K$).}
The restriction of $\pi_{\mu,\delta}$ to the subgroup $K$ of $G_1$
is decomposed into a discrete direct sum of pairwise inequivalent
representations:
\begin{alignat*}{2}
\pi_{\mu,\delta}|_K &\simeq \sideset{}{^\oplus}
\sum_{\substack{\alpha,\beta\in\mathbb{N}\\
                \alpha+\beta\equiv\delta\bmod 2}}
{\mathcal{H}}^{\alpha,\beta}(\C^n)
&&
(m \ge 1),
\\
\pi_{\mu,\delta}|_K &\simeq \sideset{}{^\oplus}
\sum_{\substack{\alpha\in\mathbb{N}\\
                \alpha\equiv\delta\bmod 2}}
{\mathcal{H}}^{\alpha,0}(\C)
\oplus
\sideset{}{^\oplus}
\sum_{\substack{\beta\in\mathbb{N}_+\\
                \beta\equiv\delta\bmod 2}}
{\mathcal{H}}^{0,\beta}(\C)
&\quad&
(m = 0).
\end{alignat*}
\end{lem}
We shall refer to ${\mathcal{H}}^{\alpha, \beta}(\C^n)$
 as a $K$-type of the representation $\pi_{ \mu, \delta}$.

The restriction $G_1\downarrow K$ is multiplicity free. Therefore any $K$-inter\-twining operator (in particular, any $G_1$-intertwining operator)
acts as a scalar on every $K$-type by Schur's lemma. We give an explicit formula of this scalar for the Knapp--Stein intertwining operator:
$$
\mathcal T_{\mu,\delta}:V_{-\mu,\delta}\to V_{\mu,\delta},
$$
which is defined as the meromorphic continuation of the following integral operator
$$
(\mathcal{T}_{\mu,\delta} f)(\eta)
:=
\int_{S^{2n-1}} f(\xi)
\left\vert\omega(\xi,\eta)\right\vert|^{-\mu-n}\left({\mathrm{sgn}}\,\omega(\xi,\eta)\right)^\delta d\sigma(\xi).
$$
Here $d\sigma $ is the Euclidean measure on the unit sphere. Further, we normalize it by
\begin{equation}\label{eqn:nKS}
\widetilde{\mathcal{T}}_{\mu,\delta}:=\frac1{C_{2n}(\mu,\delta)}\mathcal{T}_{\mu,\delta},
\end{equation}
where
$$
C_{2n}(\mu,\delta):=2\pi^{\mu+n-\frac12}\times\left\{\begin{array}{cl}
\frac{\Gamma\left(\frac{1-\mu-n}2\right)}{\Gamma\left(\frac{\mu+n}2\right)}&
(\delta=0), \\
-i\frac{\Gamma\left(\frac{2-\mu-n}2\right)}
{\Gamma\left(\frac{\mu+n+1}2\right)}&(\delta=1).
                                                    \end{array}\right.
$$

\begin{prop}\label{prop:eigen} For $\alpha,\beta\in\mathbb N$, we set $\delta\equiv\alpha+\beta$ mod $2$. The normalized Knapp--Stein intertwining operator
$\widetilde{\mathcal{T}}_{\mu,\delta}$ acts on $\mathcal
H^{\alpha,\beta}(\C^n)$ as the following scalar
$$
(-1)^\beta\pi^{-\mu}\frac{\Gamma\left(\frac{\alpha+\beta+\mu+n}2\right)}{\Gamma\left(\frac{\alpha+\beta-\mu+n}2\right)}.
$$
\end{prop}
\begin{proof} See \cite[Theorem 2.1]{CKOP} for $\delta=0$. The proof for $\delta=1$ works as well by using Lemma \ref{lem:22}.\end{proof}
\begin{rem} Without normalization, the Knapp--Stein intertwining operator $\mathcal T_{\mu,\delta}$ acts on $\mathcal H^{\alpha,\beta}(\C^n)$
as
$$
\mathcal{T}_{\mu,\delta} \Big|_{\mathcal{H}^{\alpha,\beta}(\mathbb{C}^n)}
= (-1)^\beta A_{\alpha+\beta}(\mu) \operatorname{id},
$$
where $\delta\equiv\alpha+\beta$ mod $2$ and
$$
A_k(\mu):=2\pi^{n-\frac12}\frac{\Gamma\left(\frac{k+\mu+n}2\right)}
{\Gamma\left(\frac{k-\mu+n}2\right)}\times\left\{
\begin{array}{cl}
\frac{\Gamma\left(\frac{1-\mu-n}2\right)}{\Gamma\left(\frac{\mu+n}2\right)}& (k\in 2\mathbb N), \\
-i\frac{\Gamma\left(\frac{2-\mu-n}2\right)}{\Gamma\left(\frac{\mu+n+1}2\right)}&
(k\in 2\mathbb N+1),
\end{array} \right.
$$
\end{rem}

The symplectic Fourier transform ${\mathcal F}_{\mathrm{symp}}$, defined by (\ref{eqn:Fsymp}), may be written as:
$$
({\mathcal F}_{\mathrm{symp}} f)(Y)
= \int_{\mathbb{R}^{2n}} f(X) e^{-2\pi i\omega(X,Y)} dX=\left(\mathcal F_{\R^{2n}}f\right)(JY),
$$
where $J:\R^{2n}\to\R^{2n}$ is given by $J(x,\xi):=(-\xi,x)$.

For generic complex parameter $\mu$ (\emph{e.g.} $\mu\neq n,n+2,\ldots$ for $\delta=0$), the space $V_{\mu,\delta}^\infty$ of homogeneous functions on $\R^{2n}\setminus\{0\}$
 may be regarded as a subspace of the space $\mathcal S'(\R^{2n})$ of tempered
distributions, and we have the following commutative diagram:
\[
\begin{array}{rcccc}
   \mathcal{F_{\mathrm{symp}}} & :
      &\mathcal{S}'(\mathbb{R}^{2n}) & \stackrel{\sim}{\longrightarrow} & \mathcal{S}'(\mathbb{R}^{2n})
\\[1ex]
      && \cup   & \circlearrowright & \cup
\\[1ex]
      && V_{-\mu,\delta}  & \stackrel{\sim}{\longrightarrow} & V_{\mu,\delta}
\end{array}
\]
\begin{lem}\label{lem:22}
As operators that depend meromorphically on $\mu$, $\widetilde{\mathcal{T}}_{\mu,\delta}$ satisfy the following identity:
$$
\widetilde{\mathcal{T}}_{\mu,\delta}={{\mathcal F}_{\mathrm{symp}}}|_{V_{-\mu,\delta}}.
$$
\end{lem}
\begin{proof}
The proof parallels that of \cite[Proposition 2.3]{CKOP}.
For 
$h \in C^\infty(S^{2n-1})_\delta$,
we define a homogeneous function
$h_{\mu-n} \in V_{-\mu,\delta}^\infty$
by
\[
h_{\mu-n}(r \xi)
:= r^{\mu-n} h(\xi)
\quad
(r>0, \  \xi \in S^{2n-1}).
\]
Then we recall from \cite[Proposition 2.2]{CKOP} the following formula:
\[
\mathcal{F}_{\mathbb{R}^{2n}} h_{\mu-n}(s\eta)
= \frac{\Gamma(\mu+n) e^{-\frac{\pi i}{2}(\mu+n)}}
       {(2\pi)^{\mu+n} s^{\mu+n}}
  \int_{S^{2n-1}} (\langle \xi,\eta\rangle-i0)^{-\mu-n}
  h(\xi) d\sigma(\xi),
\]
where
$(\langle \xi,\eta\rangle - i0)^\lambda$
is a distribution of $\xi, \eta$,
obtained by the substitution of
$t = \langle \xi,\eta \rangle$
into the distribution
$(t-i0)^\lambda$ of one variable $t$.

To conclude, 
we use
\begin{eqnarray*}
&&(t-i0)^{-\mu-n}
\\
&=&e^{\frac\pi2i(\mu+n)}\left(\cos\frac{\pi(\mu+n)}2|t|^{-\mu-n}-i\sin\frac{\pi(\mu+n)}2|t|^{-\mu-n}{\mathrm{sgn}}\,t\right)\\
&=& \pi
e^{\frac\pi2i(\mu+n)}\left(\frac{|t|^{-\mu-n}}{\Gamma\left(\frac{1+\mu+n}2\right)\Gamma\left(\frac{1-\mu-n}2\right)}-
i\frac{|t|^{-\mu-n}{\mathrm{sgn}}\,t}{\Gamma\left(\frac{\mu+n}2\right)\Gamma\left(\frac{2-\mu-n}2\right)}\right).
\end{eqnarray*}
\end{proof}
We note that the Knapp--Stein intertwining operator induces a unitary equivalence of representations
$\pi_{i\lambda,\delta}$ and $\pi_{-i\lambda,\delta}$ of $G_1=Sp(n,\R)$:
\begin{equation}
\label{eqn:Spdual}
\pi_{i\lambda,\delta} \simeq \pi_{-i\lambda,\delta},
\quad\text{for any $\lambda \in \R$ and $\delta \in \Z/2\Z$}.
\end{equation}

\vskip10pt

\section{Algebraic Knapp--Stein intertwining operator}\label{sec:6}


We introduce yet another model $\mathcal U_{\mu,\delta}\simeq L^2(\R^{2m+1})$,
referred to as the \emph{non-standard model}, of
the representation $\pi_{\mu,\delta}$ as the image of the partial Fourier transform
$$
\mathcal F_\xi:\, L^2(\R^{1+m+m})\overset{\sim}{\to}L^2(\R^{1+m+m}),
$$
where $\xi$ denotes the last variable in $\R^m$. Then the space $\mathcal U_{\mu,\delta}$ inherits a $G_1$-module structure from $(\pi_{\mu,\delta}, V_{\mu,\delta})$
through $\mathcal F_\xi\circ\mathcal F_t$ (see Figure \ref{fig:3models}).

The advantage of this model is that the Knapp--Stein intertwining operator becomes an algebraic operator (see Theorem \ref{thm:KNalg} below). The price to pay is that the Lie algebra $\mathfrak k$ acts on $\mathcal U_{\mu,\delta}$ by second order differential
operators. We can still give an explicit form of minimal $K$-types on the model $\mathcal U_{\mu,\delta}$ when it splits into two irreducible components
($\mu=0,\delta=0,1$) by means of $K$-Bessel functions (Section\ref{sec:minK}).

We define an endomorphism of $L^2(\R^{2m+1})$ by
\begin{equation}\label{eqn:TH}
(T_{\mu,\delta}H)(\rho,x,\eta):=\left|\frac\rho2\right|^{-\mu}({\mathrm{sgn}}\rho)^\delta H\left(\rho,\frac2\rho\eta,\frac\rho2 x\right).
\end{equation}

Regarding $\widetilde{\mathcal T}_{\mu,\delta}$ as an operator on the  $N$-picture, we have
\begin{thm}[algebraic Knapp--Stein intertwining operator]\label{thm:KNalg}
For any $\mu\in\C$ and $\delta\in\Z/2\Z$, the following diagram commutes:
$$
\xymatrix{V_{-\mu,\delta}\ar[r]^{\widetilde{\mathcal T}_{\mu,\delta}}\ar[d]_{\mathcal F_\xi\mathcal F_t}&V_{\mu,\delta}
\ar[d]^{\mathcal F_\xi\mathcal F_t}\\
\mathcal U_{-\mu,\delta}\ar[r]^{T_{\mu,\delta}}&\mathcal U_{\mu,\delta}
}
$$
\end{thm}
To prove Theorem \ref{thm:KNalg}, we work on the ambient space $\R^{2n}(=\R^{2m+2})$.
 Let
 $\mathcal F_{\R^n}$ denote the partial Fourier transform
of the last $n$ coordinates in $\R^{2n}$.
\begin{lem}\label{lem:Fn}
\begin{enumerate}
\item[1)] For $f\in V_{-\mu,\delta}$, the function $\mathcal F_{\R^n}f$ satisfies
$$
(\mathcal F_{\R^n}f)(r x,r^{-1}\eta)=|r|^\mu({\mathrm{sgn}}\,r)^\delta(\mathcal F_{\R^n}f)(x,\eta),\quad r\in\R^\times,\,x,\eta\in\R^n.
$$
\item[2)] For $f\in\mathcal S'(\R^{2n}),x,\eta\in\R^n$, we have
$$
(\mathcal F_{\R^n}\circ{\mathcal F}_{\mathrm{symp}}\circ\mathcal F_{\R^n}^{-1})f(x,\xi)=f(\xi,x).
$$
\end{enumerate}
\end{lem}

\begin{proof}[Proof of Lemma \ref{lem:Fn}]
1) This is a straightforward computation.

2) For $f(x,\xi') \in \mathcal{S}(\mathbb{R}^{2n})$,
\begin{align*}
&(\mathcal{F}_{\mathbb{R}^n} \circ \mathcal{F}_{\mathrm{symp}}
  \circ \mathcal{F}_{\mathbb{R}^n}^{-1} f) (y,\eta')
\\
={}
& \int_{\mathbb{R}^n} \int_{\mathbb{R}^{2n}} \int_{\mathbb{R}^n}
  f(x,\xi') e^{2\pi i \langle \xi,\xi'\rangle}
  e^{-2\pi i(\langle \xi,y\rangle-\langle x,\eta\rangle)}
  e^{-2\pi i \langle\eta,\eta'\rangle}
  d \xi' dx d\xi d\eta
\\
={}
& \int_{\mathbb{R}^n\times\mathbb{R}^{2n}\times\mathbb{R}^n}
  f(x,\xi') e^{2\pi i \langle \xi'-y,\xi \rangle}
  e^{-2\pi i \langle \eta'-x,\eta \rangle}
  d\xi' dx d\xi d\eta
\\
={}
& \int_{\mathbb{R}^n} f(x,\xi') \delta(\xi'-y) \delta(\eta'-x) dx d\xi'
\\
={}
& f(\eta',y).
\end{align*}
\end{proof}

From now $x,\xi,\eta$ will stand again for elements of $\R^m$, where $m=n-1$.
\begin{proof}[Proof of Theorem \ref{thm:KNalg}]According to the choice of the isomorphism (\ref{eqn:hn}) between the Lie group $N$ and the standard Heisenberg Lie group, for $f\in V_{-\mu,\delta}$, we set
\begin{eqnarray*}
  F(t,x,\xi) &:=&f(1,x,2t,\xi), \\
  H(\rho,x,\eta) &:=& (\mathcal F_t\mathcal F_\xi F)(\rho,x,\eta),
\end{eqnarray*}
where $t,\rho\in\R$ and $ x,\xi\in\R^m$.
Then $H(\rho,x,\eta)=\frac12(\mathcal F_{\R^n}f)(1,x,\frac\rho2,\eta).$ Thus, according to Lemma \ref{lem:Fn},
\begin{eqnarray*}
  \mathcal F_t\mathcal{F}_\xi({\mathcal F}_{\mathrm{symp}}f)(\rho,x,\eta) &=& \frac12\mathcal{F}_{\R^n}({\mathcal F}_{\mathrm{symp}}f)(1,x,\frac\rho2,\eta) \\
   &=& \frac12(\mathcal{F}_{\R^n}f)(\frac\rho2,\eta,1,x) \\
   &=& \frac12\left|\frac\rho2\right|^{-\mu}({\mathrm{sgn}}\rho)^\delta(\mathcal{F}_{\R^n}f)(1,\frac2\rho\eta,\frac\rho2,\frac\rho2x) \\
  &=& \left|\frac\rho2\right|^{-\mu}({\mathrm{sgn}}\rho)^\delta H(\rho,\frac2\rho\eta,\frac\rho2x).
\end{eqnarray*}
Now Theorem follows from Lemma \ref{lem:22}.
\end{proof}
\vskip10pt

\section{Minimal $K$-type in a non-standard model}\label{sec:minK}

We give an explicit formula for two particular $K$-finite vectors of $\pi_{0,\delta}$
(in fact, minimal $K$-types of irreducible components
$\pi_{0,\delta}^\pm$ of $\pi_{0,\delta}$;
see Theorem \ref{thm:twosum} 1)) in the 
non-standard $L^2$-model $\mathcal U_{0,\delta}(\simeq L^2(\R^{2m+1}))$.
The main results (see Proposition \ref{prop:FFF}) show that minimal $K$-types are represented in terms of $K$-Bessel functions in this model. Although we do not use these results in the proof
of Theorem \ref{thm:twosum}, we think they are interesting of their own from the view point of geometric analysis of small representations. It is noteworthy that similar feature to Proposition \ref{prop:FFF} has been observed in the $L^2$-model of minimal representations of some other reductive groups (see e.g. \cite[Part III]{KobOr2}).

We begin with the identification
 $$
\C\overset{\sim}{\to}\mathcal H^{0,0}(\C^{m+1}),\quad 1\mapsto \bold{1}\quad
\text{(constant function)},
$$
and extend it to a homogeneous function on $\R^{2n}$ belonging to
$V_{0,0}$ (see \eqref{eqn:homoRN}). Using the formula
(\ref{eqn:Npic}) in the $N$-picture, 
we set 
$$h^+(t,x,\xi):=(1+4t^2 + |x|^2 + |\xi|^2)^{-\frac{m+1}{2}}.$$
Notice that $h^+(t,x,\xi)\in V_{0,0}\cap\mathcal H^{0,0}(\C^{m+1})$
in the $K$-type formula of $\pi_{0,0}$ (see Lemma \ref{lem:Kdeco}).

Let 
\begin{equation}\label{eqn:phipsi}
\psi(\rho,x,\eta)
 :=
 \left(1+|x|^2\right)^{\frac{1}{2}}
 \Bigl(\frac{\rho^2}{4}+|\eta|^2\Bigr)^{\frac{1}{2}}.
\end{equation}
Likewise we identify
\[
\mathbb{C}^{m+1}
\overset{\sim}{\to}
\mathcal{H}^{0,1}(\mathbb{C}^{m+1}),
\ 
b \mapsto \sum_{j=0}^m b_j \overline{z}_j,
\]
and set
\begin{equation}\label{eqn:hb}
h_b^-(t,x,\xi)
:= (1+4t^2 + |x|^2 + |\xi|^2)^{-\frac{m+2}{2}}
   (b_0(1-2it) + \sum_{j=1}^m b_j(x_j-i\xi_j)),
\end{equation}
\begin{equation}\label{eqn:phib}
\varphi_b(\rho,x,\eta)
:= \omega( \begin{pmatrix} (1+|x|^2)^{\frac{1}{2}} \\
                           (\frac{\rho^2}{4}+|\eta|^2)^{\frac{1}{2}}
           \end{pmatrix},
   b_0 \begin{pmatrix}  1 \\ \frac{\rho}{2}  \end{pmatrix}
   + \sum_{j=1}^m b_j  \begin{pmatrix} x_j \\ \eta_j \end{pmatrix}
         ),
\end{equation}
where $\omega$ denotes the standard symplectic form on $\mathbb{C}^2$
defined as in \eqref{eqn:symp}.
Then
$
h_b^- \in V_{0,1} \cap \mathcal{H}^{0,1} (\mathbb{C}^{m+1})
$
in the $K$-type formula of $\pi_{0,1}$ (see Lemma \ref{lem:Kdeco}).

Let $K_\nu(z)$ denote the modified Bessel function of the second
kind ($K$-Bessel function for short).
Then the $K$-finite vectors $h^+$ and $h^-_b$ ($b\in\C^{m+1}$) in the
standard model ($N$-picture) are of the following form in the non-standard
model $\mathcal U_{0,\delta}$.

\begin{prop}\label{prop:FFF}
\
\begin{enumerate}
\item[1)]
$\displaystyle
(\mathcal{F}_t\mathcal{F}_\xi h^+)(\rho,x,\eta)
 =
 \frac{\pi^{\frac{m+2}{2}}}{\Gamma(\frac{m+1}{2})}
 \,
 K_0 \left(2\pi\psi(\rho,x,\eta)\right)$.
\item[2)]
$\displaystyle
(\mathcal{F}_t\mathcal{F}_\xi h_b^-)(\rho,x,\eta)
 =
 \frac{\pi^{\frac{m+2}{2}}}{2\Gamma(\frac{m+2}{2})}
 \,
 \frac{\varphi_b(\rho,x,\eta)}{\psi(\rho,x,\eta)}
 \,
 \exp \left(-2\pi\psi(\rho,x,\eta)\right)$.
\end{enumerate}
\end{prop}

The rest of this section is devoted to the proof of Proposition \ref{prop:FFF}.
 In order to get simpler formulas we also use the following
normalization $\widetilde K_\nu(z):=\left(\frac z2\right)^{-\nu}K_\nu(z)$ \cite[Section 7.2]{KM-intopq}.

\begin{lem}\label{lem:Imu}
For every $\mu\in\R$ let us define the following function on $\R\times\R^m$ :
$$
I_\mu\equiv I_\mu(a,\eta):= \int_{\R^m}(a^2+|\xi|^2)^{-\mu} e^{-2i\pi \langle \xi,\eta\rangle}d\xi.
$$
Then,
\begin{equation}\label{eqn:Imu}
I_\mu(a,\eta)=\frac{2\pi^{\frac m2}}{\Gamma(\mu)} {a}^{ m-2\mu}\widetilde K_{\frac m2-\mu}(2\pi a|\eta|).
\end{equation}
\end{lem}
\begin{proof}

Recall the classical Bochner formula
$$
\int_{S^{m-1}}e^{-2i\pi s\langle\xi,\xi'\rangle}d\sigma(\xi)
=2\pi s^{1-\frac m2} J_{\frac m2-1}(2\pi s), \quad{\mathrm{for}}\,\xi'\in S^{m-1},
$$
where $J_\nu(z)$ denotes the Bessel function of the first kind.
Then,
\begin{eqnarray*}
  I_\mu(a,\eta)
  &=& \int_0^\infty\int_{S^{m-1}} (a^2+r^2)^{-\mu}e^{-2i\pi r|\eta|\langle\xi,\frac{\eta}{|\eta|}\rangle} r^{m-1}dr d\sigma(\xi)\\
  &=& 2\pi |\eta|^{1-\frac m2}\int_0^\infty r^{\frac m2} J_{\frac m2-1}(2\pi r|\eta|)(r^2+a^2)^{-\mu}dr.
\end{eqnarray*}

According to \cite[8.5 (20)]{E} we have
$$
\int_0^\infty x^{\nu+\frac12}(x^2+a^2)^{-\mu-1}J_\nu(xy)(xy)^{\frac12}dx=
\frac{a^{\nu-\mu}\,y^{\mu+\frac12}\,K_{\nu-\mu}(ay)}{2^{\mu}\,\Gamma(\mu+1)},
$$
for $\operatorname{Re}a>0$ and
$-1<\operatorname{Re}\nu<2\operatorname{Re}\mu+\frac{3}{2}$,
which implies
\begin{eqnarray*}
I_\mu(a,\eta)&=&\frac{2\pi^{\mu}}{\Gamma(\mu)}\left(\frac {a}{|\eta|}\right)^{\frac m2-\mu}K_{\frac m2-\mu}(2\pi a|\eta|)\\
&=&
\frac{2\pi^{\frac m2}}{\Gamma(\mu)} {a}^{ m-2\mu}\widetilde K_{\frac m2-\mu}(2\pi a|\eta|).
\end{eqnarray*}
\end{proof}

In particular, we have
\begin{align*}
&I_{\frac{m+1}{2}}(a,\eta)
 =
 \frac{\pi^{\frac{m+2}{2}}}{\Gamma(\frac{m+1}{2})}
 \,
 \frac{\exp(-2\pi a|\eta|)}{a},
\\
&I_{\frac{m+2}{2}}(a,\eta)
 =
 \frac{2\pi^{\frac{m+2}{2}}}{\Gamma(\frac{m+2}{2})}
 \,
 \frac{|\eta|}{a}
 \,
 K_1(2\pi a |\eta|).
\end{align*}
 Here we used
$\widetilde K_{-\frac12}(z)=\frac{\sqrt{\pi}}2 e^{-z}$ in the first identity.
By a little abuse of notation,
we write $h_{(0)}^-$ and $h_{(1)}^-$ for
$h_{(1,0,\dots,0)}^-$
and
$h_{(0,1,0,\dots,0)}^-$,
respectively.
\begin{lem}\label{lem:F01} For $(t,x)\in\R\times\R^m$,
we set
$$
a\equiv a(t,x) := \sqrt{1+4t^2+|x|^2}.
$$
Then,
\begin{eqnarray*}
&&(\mathcal F_\xi h^+)(t,x,\eta)=
I_{\frac{m+1}{2}} (a(t,x),\eta),
\\
&&(\mathcal F_\xi h_{(1)}^-)(t,x,\eta)
=\Bigl(x_1+\frac{1}{2\pi} \frac{\partial}{\partial\eta_1}\Bigr)
 I_{\frac{m+2}{2}} (a(t,x),\eta).
\end{eqnarray*}
\end{lem}

\begin{proof}

By definition
\begin{equation*}
(\mathcal F_\xi h^+)(t,x,\eta)=\int_{\R^m}
({1+4t^2+|x|^2}+|\xi|^2)^{-\frac{m+1}2}
e^{-2\pi i\langle\xi,\eta\rangle}
d\xi
\end{equation*}

\begin{eqnarray*}
( \mathcal F_\xi
h_{(1)}^- )(t,x,\eta)&=& \int_{\R^m}
({1+4t^2+|x|^2}+|\xi|^2)^{-\frac{m+2}2}(x_1-i\xi_1)
e^{-2\pi i\langle\xi,\eta\rangle}d\xi\\
&=&
\Bigl(x_1+\frac{1}{2\pi}\frac{\partial}{\partial\eta_1}\Bigr) I_{\frac{m+2}{2}}
\left(\sqrt{1+4t^2+|x|^2},\eta\right).
\end{eqnarray*}
Hence Lemma \ref{lem:F01} is proved.
\end{proof}

\begin{proof}[Proof of Proposition \ref{prop:FFF}]
We recall from \cite[vol.~I, 1.4 (27); 1.13 (45); 2.13 (43)]{E}
 the following formulas:
For $\operatorname{Re}d>0$,
$\operatorname{Re}c>0$
and $s>0$, 
\begin{align}
&\int_0^\infty
 \frac{\exp(-d(t^2+c^2)^{\frac{1}{2}})}
      {(t^2+c^2)^{\frac{1}{2}}}
 \,
 \cos(st)dt
\label{eqn:ec}
=
 K_0(c(s^2+d^2)^{\frac{1}{2}}).
\\
&
\nonumber
\\
&\int_0^\infty
 \frac{K_\nu(d(t^2+c^2)^{\frac{1}{2}})}
      {(t^2+c^2)^{\frac{\nu}{2}}}
 \,
 \cos(st)dt
\label{eqn:Kc}
= \sqrt{\frac{\pi}{2}}
   \,\frac{
   K_{\nu-\frac{1}{2}}
   (c(s^2+d^2)^{\frac{1}{2}})}{d^{\nu}
   c^{\nu-\frac{1}{2}} (s^2+d^2)^{\frac{1}{4}-\frac{1}{2}\nu}}
\\
&
\phantom{\int_0^\infty
 \frac{K_\nu(d(t^2+c^2)^{\frac{1}{2}})}
      {(t^2+c^2)^{\frac{\nu}{2}}}
 \,
 \cos(st)dt
= \sqrt{\frac{\pi}{2}}
   \,\frac{
   K_{\nu-\frac{1}{2}}
   (c(s^2+d^2)^{\frac{1}{2}})}{d^{\nu}
   c^{\nu-\frac{1}{2}} (s^2+d^2)^{\frac{1}{4}-\frac{1}{2}\nu}}
}
\llap{$
= 2^{\nu-1} \sqrt{\pi} \, d^{-\nu} \, c^{1-2\nu}
   \,
   \widetilde{K}_{-\nu+\frac{1}{2}}
   (c(s^2+d^2)^{\frac{1}{2}}).
$}
\nonumber
\\
&
\nonumber
\\
&\int_0^\infty
 \frac{tK_1(d(t^2+c^2)^{\frac{1}{2}})}
      {(t^2+c^2)^{\frac{1}{2}}}
 \sin(st)dt
 = \frac{\pi s}{2d}
   \frac{\exp(-c(s^2+d^2)^{\frac{1}{2}})}
        {(s^2+d^2)^{\frac{1}{2}}}.
\label{eqn:Ks}
\end{align}

We apply the formulas \eqref{eqn:ec} and \eqref{eqn:Kc} with
$d = 4\pi|\eta|$,
$c = \frac{1}{2}(1+|x|^2)^{\frac{1}{2}}$
and $s = 2\pi\rho$.
In view that
$a \equiv a(t,x) = 2(t^2+c^2)^{\frac{1}{2}}$
and
$2\pi\psi(\rho,x,\eta) = c(s^2+d^2)^{\frac{1}{2}}$,
we get
\begin{align*}
&\int_{-\infty}^\infty
 \frac{\exp(-2\pi a|\eta|)}{a}
 \,
 e^{-2\pi i t \rho} dt
 =
 K_0(2\pi \psi(\rho,x,\eta)),
\\
&\int_{-\infty}^\infty
 \frac{K_1(2\pi a|\eta|)}{a}
 \,
 e^{-2\pi i t\rho} dt
 =
 \frac{1}{4|\eta|(1+|x|^2)^{\frac{1}{2}}}
 \,
 \exp(-2\pi\psi(\rho,x,\eta)).
\end{align*}
Here, we have used again
$\widetilde{K}_{-\frac{1}{2}}(z)
 = \frac{\sqrt{\pi}}{2} \, e^{-z}$
for the second equation.
Thus the first statement has been proved.

To see the second statement,
it is sufficient to treat the following two cases:
$b = (1,0,\dots,0)$
and
$b = (0,1,0,\dots,0)$.
We use
\begin{align*}
\mathcal{F}_t\mathcal{F}_\xi h_{(1)}^-
&= \mathcal{F}_t \left(\Bigl(x_1+\frac{1}{2\pi}
   \, \frac{\partial}{\partial\eta_1}\Bigr)
   I_{\frac{m+2}{2}} (a(t,x),\eta)\right)
\\
&= \Bigl(x_1+\frac{1}{2\pi}
   \, \frac{\partial}{\partial\eta_1}\Bigr)
   \mathcal{F}_t
   \Bigl(I_{\frac{m+2}{2}} (a(t,x),\eta)\Bigr).
\end{align*}
Now use
\[
\Bigl(x_1+\frac{1}{2\pi} \, \frac{\partial}{\partial\eta_1}\Bigr)
\frac{\exp(-2\pi\psi(\rho,x,\eta))}{(1+|x|^2)^{\frac{1}{2}}}
=
\frac{\varphi_{(1)}(\rho,x,\eta) \exp(-2\pi\psi(\rho,x,\eta))}
     {\psi(\rho,x,\eta)}.
\]
The case $b = (1,0,\dots,0)$ goes similarly by using the formula
\eqref{eqn:Ks}. 
\end{proof}

\section{Branching law for $GL(2n,\R)\downarrow Sp(n,\mathbb R)$}\label{sec:prf}

From now we give a proof of Theorem \ref{thm:global} with emphasis on geometric analysis involved.

Our strategy is the following. Suppose $P$ is a closed subgroup of a Lie group $G$,
$\chi: P \to \C^\times$ a unitary character, and $\mathcal{L} := G
\times_P \chi$
 a $G$-equivariant line bundle
over $G/P$. We write $L^2(G/P,\mathcal{L})$ for the Hilbert space
consisting of $L^2$-sections for the line bundle $\mathcal{L}
\otimes (\Lambda^{top} T^*(G/P))^{\frac{1}{2}}$. Then the group $G$
acts on $L^2(G/P,\mathcal{L})$ as a unitary representation, to be
denoted by $\pi_\chi^G$, by translations. 

If $(G,H)$ is a reductive
symmetric pair and $P$ is a parabolic subgroup of $G$, then there
exist finitely many open $H$-orbits $\mathcal{O}^{(j)}$ on the real
flag variety $G/P$ such that ${\cup}_j \mathcal{O}^{(j)}$ is open
dense in $G/P$. (In our cases below, the number of open $H$-orbits
is at most two.) Applying the Mackey theory, we see that the
restriction of the unitary representation $\pi_\chi^G$ to the
subgroup $H$ is unitarily equivalent to a finite direct sum:
\[
\pi_\chi^G|_H \simeq \bigoplus_j L^2(\mathcal{O}^{(j)},
\mathcal{L}|_{\mathcal{O}^{(j)}}).
\]
Thus the branching problem is reduced to the irreducible
decomposition of
$L^2(\mathcal{O}^{(j)},\mathcal{L}|_{\mathcal{O}^{(j)}})$,
equivalently, the Plancherel formula for the homogeneous line bundle
$\mathcal{L}|_{\mathcal{O}^{(j)}}$ over open $H$-orbits $\mathcal{O}^{(j)}$.

In our specific setting, where $G = GL(N,\R)$ and $P=P_N$ (see (\ref{eqn:PN})), 
the base space $G/P$
is the real projective space $\mathbb{P}^{N-1}\R$. 
For
$(\lambda,\delta) \in \R \times \Z/2\Z$, we define a unitary character $\chi_{i\lambda,\delta}$
of $P_N$ by
$$
\chi_{i\lambda,\delta}
\left(
\begin{matrix}
a & {}^t b\\
0 & C
\end{matrix}
\right)
:=\vert a\vert^\lambda(\operatorname{sgn} a)^\delta,\; a\in GL(1,\R),\,C\in GL(N-1,\R),\, b\in
\R^{N-1},
$$
in the matrix realization of $P_N$. 
Then $\pi_{\chi_{i\lambda,\delta}}^G$
coincides with $\pi_{{i\lambda,\delta}}^G$ in previous notation.
In this and the next three sections, we find the explicit
irreducible decomposition of
$L^2(\mathcal{O}^{(j)},\mathcal{L}|_{\mathcal{O}^{(j)}})$ with respect to $\pi_{i\lambda,\delta}^G$.

We begin with the case $H=G_1$, \emph{i.e.} 
$$
(G,H)\equiv\left( GL(2n,\R),\, Sp(n,\R)\right).
$$
As we have already seen in Section \ref{sec:4} the group $G_1$ acts transitively on $G/P_N$, and
we have the following unitary equivalence of unitary representations of $G_1=Sp(n,\R)$:
$$
\pi_{{i\lambda,\delta}}^G\vert_{G_1}\simeq \pi_{{i\lambda,\delta}}^{G_1}.
$$
Here $\pi_{{i\lambda,\delta}}^{Sp(n,\R)}$ is a unitary representation of $Sp(n,\R)$
induced from the maximal parabolic subgroup $P=G_1\cap P_N\simeq
\left( GL(1,\R)\times Sp(n-1,\R)\right)\ltimes H^{2n-1}$.

Thus the following two statements are equivalent.

\begin{thm}\label{thm:2}
The restriction of $\pi_{i\lambda,\delta}^{GL(2n,\R)}$ from
$GL(2n,\R)$ to $Sp(n,\R)$ stays irreducible for any $\lambda\in\R^\times$ and
$\delta\in\{0,1\}$. It splits into two irreducible components for $\lambda=0,\delta=0,1$ and $n\geq2$.
\end{thm}

\begin{thm}\label{thm:1.1}
Let $P$ be a maximal parabolic subgroup of $G_1$ whose Levi part is isomorphic to $GL(1,\R)\times Sp(n-1,\R)$, and denote by $\pi_{i\lambda, \delta}$ ($\lambda\in\R,\delta=0,1$) the corresponding unitary (degenerate) principal
series representation of $G_1$. Then for
 $n\geq 2$,
 $\pi_{i\lambda,\delta}$ are irreducible for any
$(\lambda,\delta)\in\mathbb R^\times\times \mathbb Z/2\mathbb Z$, and splits into a direct sum of two irreducible components
for $\lambda=0,\delta=0,1$.
\end{thm}
Theorem \ref{thm:1.1} itself was proved in \cite[Theorem 7.3]{KOPU}. The case of $\delta=0$ was studied by different methods
earlier in \cite{Farmer} and also very recently in \cite{Barb} ($\lambda=0$ and $\delta=0$)  in the context of special unipotent representations of the split group $Sp(n,\R)$. We give yet another proof of Theorem \ref{thm:1.1} in the most interesting case, i.e. in the case $\lambda=0$ and $\delta=0,1$ below. 

Theorem \ref{thm:twosum} describes a finer structure of the
irreducible summands.
The novelty here (even for the $\delta=0$ case) is that we characterize explicitly
 the two irreducible summands by their $K$-module structure, and
 also by their $P$-module structure. The former is given in terms of complex spherical harmonics 
 (\emph{cf}.\ Lemma
 \ref{lem:Kdeco}) and the latter in terms of Hardy spaces (\emph{cf}.\ Lemma \ref{lem:1.8}), as follows:
\begin{thm}\label{thm:twosum}
Let $n\geq2$ and $\delta\in\Z/2\Z$. The unitary representation $\pi_{0,\delta}$
of $G_1=Sp(n,\R)$ splits into the direct sum of two irreducible representations
of $G_1$:
\begin{equation}\label{eqn:Pdeco}
    \pi_{0,\delta} = \pi_{0,\delta}^+ \oplus \pi_{0,\delta}^-.
\end{equation}
1) (Characterization by $K$-type).
Each irreducible summand in (\ref{eqn:Pdeco}) has the following $K$-type formula:
\begin{eqnarray*}
\pi_{0,\delta}^+
&\simeq&\sideset{}{^\oplus}
\sum_{\substack{\beta\in2\mathbb{N}\\
                \alpha\equiv\beta+\delta\bmod 2}}\mathcal H^{\alpha,\beta}(\C^n),
 \\
\pi_{0,\delta}^-
&\simeq&\sideset{}{^\oplus}
\sum_{\substack{\beta\in2\mathbb{N}+1\\
                \alpha\equiv\beta+\delta\bmod 2}}\mathcal H^{\alpha,\beta}(\C^n),
\end{eqnarray*}
where $\sum^\oplus$ denotes the Hilbert completion of the algebraic  direct sum.

2) (Characterization by Hardy spaces). The irreducible summands $\pi_{0,\delta}^\pm$
consist of two Hardy spaces via the isomorphism (\ref{eqn:AN}):
\begin{align*}
&\pi_{0,0}^+
 \simeq \mathcal{H}_+^2(W_+) \oplus \mathcal{H}_-^2(W_+),
&\pi_{0,0}^-
 \simeq \mathcal{H}_+^2(W_-) \oplus \mathcal{H}_-^2(W_-),
\\
&\pi_{0,1}^+
 \simeq \mathcal{H}_+^2(W_+) \oplus \mathcal{H}_-^2(W_-),
&\pi_{0,1}^-
 \simeq \mathcal{H}_+^2(W_-) \oplus \mathcal{H}_-^2(W_+).
\end{align*}
Here, $W_\pm$ are the subspaces of $L^2(\mathbb{R}^{2m})$ defined in
\eqref{eqn:decompMN}, 
and $\mathcal{H}_\pm^2(W_\varepsilon)$ are the $W_\varepsilon$-valued
Hardy spaces.

3) (Characterization by the Knapp--Stein intertwining operator). The irreducible summands 
$\pi_{0,\delta}^\pm$ are the $\pm1$ eigenspaces of the normalized Knapp--Stein intertwining operator $\widetilde{\mathcal T_{0,\delta}}$ (see (\ref{eqn:nKS})).
\end{thm}
\begin{proof}
1) and 3)
The normalized Knapp--Stein intertwining operator $\widetilde{\mathcal{T}}_{0,\delta}$ has eigenvalues either $1$ or $-1$ according to the parity of the $K$-type $\mathcal H^{\alpha,\beta}(\C^n)$, namely $\beta\equiv0$ or $\beta\equiv1$ mod $2$ by Proposition \ref{prop:eigen}. 
Hence the statements 1) and 3) are proved.

2) 
In the model $\mathcal U_{0,\delta}\simeq L^2(\R^{2m+1})$
(see Section \ref{sec:6}), the Knapp--Stein intertwining operator $\widetilde{\mathcal{T}}_{0,\delta}$ is equivalent to the algebraic operator
$$
T_{0,\delta}:H(\rho,x,\eta)\to (\operatorname{sgn}\rho)^\delta H\left(\rho,\frac2\rho\eta,\frac\rho2x\right),
$$
by Theorem \ref{thm:KNalg}.

In turn,
it follows from Lemma \ref{lem:invol} that $T_{0,\delta}$ is
transfered to the operator 
\begin{equation}\label{eqn:Srho}
\mathfrak{S}(\rho,*) \mapsto (\operatorname{sgn}\rho)^\delta
\mathfrak{S}^{\dag_\rho}(\rho,*)
\end{equation}
in the operator calculus model
$L^2(\mathbb{R},\mathrm{HS}(L^2(\mathbb{R}^m),L^2(\mathbb{R}^m)))$
(see Figure \ref{fig:3models}).
In view of the $\pm 1$ eigenspaces of the transform (\ref{eqn:Srho}), we see that the statement 2) follows from the characterization of $W_\pm$ (see Lemma \ref{lem:Vpm}) and
the isomorphism $\mathcal F_t:\mathcal H_\pm^2(W_\varepsilon)\overset{\sim}{\to} L^2(\R_\pm, W_\varepsilon)$ given in Lemma \ref{lem:Hardy}.

Finally, we need to prove that the summands $\pi ^\pm_{0,\delta}$ are irreducible $G_1$-modules. This is deduced from the decomposition of $\pi_{0,\delta}^\pm$ by means of Hardy spaces in 2) and from the following lemma.
\end{proof}

\begin{lem}
For any $\delta\in\Z/2\Z$, none of the Hardy spaces $\mathcal H^2_\pm(W_\varepsilon)$ ($\varepsilon=\pm$) is $G_1$-stable with respect to $\pi_{0,\delta}$.
\end{lem}
\begin{proof}
For $Z:=(z_1,\ldots,z_m)=x+i\xi\in\C^m\simeq\R^{2m}$
(see \eqref{eqn:Npic}),
we set
\begin{align*}
&f_{0,0}(t,x,\xi)
 :=
 (1+4t^2+|x|^2+|\xi|^2)^{-\frac{m+1}{2}},
\\
&f_{0,1}(t,x,\xi)
 :=
 (1+4t^2+|x|^2+|\xi|^2)^{-\frac{m+2}{2}} \, (x_1-i\xi_1),
 \\
 &f_{1,0}(t,x,\xi)
 :=
 (1+4t^2+|x|^2+|\xi|^2)^{-\frac{m+2}{2}}\, (x_1+i\xi_1),
 \\
 &f_{1,1}(t,x,\xi)
 :=
 (1+4t^2+|x|^2+|\xi|^2)^{-\frac{m+3}{2}}\, (1+4t^2-x_1^2-\xi_1^2).
\end{align*}
We note that $f_{0,0}=h^+$ and $f_{0,1}=h^-_{(0,1,0,\ldots,0)}=h^-_{(1)}$ in the notation of Section
\ref{sec:minK}.
Then we have $f_{\alpha,\beta}\in\mathcal H^{\alpha,\beta}(\C^n)$ for any $\alpha,\beta\in\{0,1\}$.
In view of Theorem \ref{thm:twosum} 1), we get 
\begin{align*}
&f_{0,0}(t,x,\xi)
\in \mathcal H^{0,0}(\C^n)\subset V_{0,0}^+,
\\
&f_{0,1}(t,x,\xi)
 \in \mathcal H^{0,1}(\C^n)\subset V_{0,1}^-,
 \\
 &f_{1,0}(t,x,\xi)
\in \mathcal H^{1,0}(\C^n)\subset V_{0,1}^+,
 \\
 &f_{1,1}(t,x,\xi)
\in \mathcal H^{1,1}(\C^n)\subset V_{0,0}^-,
\end{align*}
where $V_{0,\delta}^\pm$ stands for the representaion space in the $N$-picture corresponding to
$\pi_{0,\delta}^\pm$ in Theorem \ref{thm:twosum}.
Suppose now that one of the Hardy spaces $\mathcal H^2_\pm(W_\varepsilon)$ were
$G_1$-stable with respect to $\pi_{0,\delta}$. 
Then its orthogonal complementary subspace for the decomposition in
Theorem \ref{thm:twosum} 2) would be also $G_1$-stable.
Since $K$-type is multiplicity-free in $\pi_{0,\delta}$ by Lemma
\ref{lem:Kdeco},
either $\mathcal{H}_\pm^2(W_\varepsilon)$ or its complementary
subspace should contain
 the $K$-type $\mathcal H^{\alpha,\beta}(\C^n)$ for some $\alpha,\beta=0$ or $1$. But this never happens because
 $f_{\alpha,\beta}(t,x,\xi)=f_{\alpha,\beta}(-t,x,\xi)$ and thus $\operatorname{supp}
 \mathcal F_t f_{\alpha,\beta}\nsubseteq\R_\pm$ (see Lemma \ref{lem:Hardy} 4)).
Thus lemma is proved.
\end{proof}

%
%
%
%
%
%
%

\begin{rem}
The case $n=1$ is well known. Here the group $Sp(1,\mathbb{R})$ is
isomorphic to $SL(2,\mathbb{R})$, and
 $\pi_{i\lambda,\delta}$ are
irreducible except for $(\lambda,\delta)=(0,1)$, while $\pi_{0,1}$
splits into the direct sum of two irreducible unitary
representations:
\begin{eqnarray*}
\pi_{0,1}^{Sp(1,\R)}&\simeq&\mathcal H^2_+(\C)\oplus \mathcal H^2_-(\C)\\
&\simeq&\left(\quad\sideset{}{^\oplus}\sum_{\alpha\in2\mathbb N+1}\mathcal H^{\alpha,0}(\C)\right)\oplus
\left(\quad\sideset{}{^\oplus}\sum_{\beta\in2\mathbb N+1}\mathcal H^{0,\beta}(\C)\right).
\end{eqnarray*}
The spaces ${\mathcal{H}}^{\alpha,0}(\C)$ and ${\mathcal{H}}^{0,\beta}(\C)$ are
one dimensional, and
\begin{align*}
(t+i)^\alpha (t^2 + 1)^{-\frac{\alpha+1}{2}}
& \in {\mathcal{H}}^{\alpha,0}(\C)\cap V_{0,1},
\\
(t-i)^\beta (t^2 + 1)^{-\frac{\beta+1}{2}}
& \in {\mathcal{H}}^{0,\beta}(\C)\cap V_{0,1}.
\end{align*}
The former function extends holomorphically to the
upper half plane $\Pi_+$, and the latter one extends holomorphically to
$\Pi_-$ if $\alpha,\beta \equiv 1 \bmod 2$, namely, if $\delta
\equiv 1$. 
\end{rem}

As formulated in Theorem \ref{thm:1.1},
 our result may be compared with general theory
 on (degenerate) principal series representations
 of real reductive groups.
For instance,
 according to Harish-Chandra and Vogan--Wallach \cite{V90},
 such representations are at most a finite sum of irreducible
 representations and are `generically' irreducible.
A theorem of Kostant \cite{Ko69}
 asserts that spherical unitary principal series representations
 (induced from minimal parabolic subgroups)
 are irreducible.

There has been also extensive research
 on the structure
 of (degenerate) principal series representations
 in specific cases,
in particular, in the case where the unipotent radical
 of $P$ is abelian
by A. U. Klimyk,  B. Gruber,
R. Howe, E.--T. Tan, S.--T. Lee, S. Sahi and others by algebraic and
combinatorial methods (see \emph{e.g.}\ \cite{HS} and
 references therein).

We have not adopted here the aforementioned methods, but have used
 the idea of branching laws to {\it non-compact subgroups}
 (see \cite{K02})
primarily because of the belief that the latter approach to very small representations
will open new aspects of the theory of  geometric analysis.
\vskip10pt

\section{Branching law for $GL(2n,\R)\downarrow GL(n,\mathbb C)$}\label{sec:glnc}
\hfill\break
Let $P_n^{\mathbb C}=L_n^\C N_n^\C$ be the standard maximal parabolic subgroup of
$GL(n,\mathbb C)$ corresponding to the partition $n=1+(n-1)$,
namely, the Levi subgroup $L_n^\C$ of $P_n^\C$ is isomorphic to
$GL(1,\mathbb C)\times
GL(n-1,\mathbb C)$ and the unipotent radical
$N_n^\C$ is the complex abelian group $\C^{n-1}$. Inducing from a unitary character
$(\nu,m)\in\R\times\mathbb Z$ of the first factor of $L_n^\C$,
$GL(1,\mathbb C)\simeq \R_+\times
S^1$ we define a degenerate principal series representation
$\pi_{i\nu,m}^{GL(n,\mathbb C)}$ of $GL(n,\C)$.
They are pairwise inequivalent, irreducible unitary representations of
$GL(n,\C)$
 (see \cite[Corollary 2.4.3]{HS}).

We identify $\C^n$ with
$\R^{2n}$, and regard
$$G_2 := GL(n,\C)$$
as a subgroup of $G = GL(2n,\R)$.

\begin{thm}[Branching law $GL(2n,\R) \downarrow GL(n,\C)$]\label{thm:3}
\begin{equation}\label{eqn:GLnC}
\pi_{i\lambda,\delta}^{GL(2n,\R)}|_{GL(n,\mathbb C)}\simeq
\sideset{}{^\oplus}\sum_{m\in2\mathbb Z+\delta} \pi_{i\lambda,m}^{GL(n,\mathbb C)}.
\end{equation}
\end{thm}
\begin{proof}
The group $G_2=GL(n,\mathbb C)$ acts transitively on
the real projective space $\mathbb P^{2n-1}\mathbb R$,
and the unique (open) orbit $\mathcal{O}_2 :=
\mathbb{P}^{2n-1}\R$
is represented as a homogeneous space $G_2/H_2$ where the isotropy group $H_2$ is of
the form
$$
H_2 \simeq (O(1) \times GL(n-1,\C))N_n^\C.
$$
Since $P_n^\C/H_2 \simeq S^1/\{\pm1\}$,
we have a $G_2$-equivariant fibration:
$$
S^1/\{\pm1\}\to \mathbb{P}^{2n-1}\R \to GL(n,\mathbb C)/P_n^{\mathbb C}.
$$
Further, if we denote by $\C_\delta$ the
one-dimensional representation of $H_2$
obtained as the following compositions:
$$
H_2 \to H_2/GL(n-1,\C)N_n^\C \xrightarrow{\delta} \C^\times,
$$
then the $G$-equivariant line bundle
$\mathcal{L}_{i\lambda,\delta} = G \times_P \C_{i\lambda,\delta}$ is
represented as a $G_2$-equivariant line bundle simply by
$$
\mathcal{L}_\delta :=
\mathcal{L}_{i\lambda,\delta} |_{\mathcal{O}_2}
\simeq GL(n,\C) \times_{H_2} \C_\delta.
$$
Therefore, we have an isomorphism as unitary representations of $G_2$:
$$
\mathcal H_{i\lambda,\delta}^{GL(2n,\R)}|_{G_2}\simeq
L^2(\mathcal O_2,\mathcal{L}_\delta).
$$
Taking the Fourier series expansion of
$L^2(\mathcal{O}_2,\mathcal{L}_\delta)$ along the fiber
$S^1/\{\pm1\}$,
we get the irreducible decomposition \eqref{eqn:GLnC}.
\end{proof}

An interesting feature of Theorem \ref{thm:3} is that the degenerate
 principal series representation $\pi_{i\lambda,\delta}^{GL(2n,\R)}$
 is discretely decomposable with respect to the restriction
$GL(2n,\R) \downarrow GL(n,\C)$.
We have seen this by finding explicit branching law, however,
discrete decomposability of the restriction
$\pi_{i\lambda,\delta}^{GL(2n,\R)}|_{GL(n,\C)}$ can
 be explained also by the general theory
\cite{xkAnn98} as follows:

Let $\mathfrak{t}$ be a Cartan subalgebra of $\mathfrak{o}(2n)$,
and we take a standard basis
$\{f_1,\dots,f_n\}$ in $i\mathfrak{t}^*$ such that the
dominant Weyl chamber for the disconnected group
$K=O(2n)$ is given as
$$
i \mathfrak{t}_+^*
= \{(\lambda_1,\dots,\lambda_n):
    \lambda_1 \ge \lambda_2 \ge \dots \ge \lambda_n \ge 0 \}.
$$
For $K_2:=G_2 \cap K \simeq U(n)$
the Hamiltonian action of $K$ on the cotangent bundle $T^*(K/K_2)$ has
 the momentum map
$T^*(K/K_2) \to i\mathfrak{k}^*$.
The intersection of its image
with the dominant Weyl chamber
$i\mathfrak{t}_+^*$ is given by
\begin{align*}
&i \mathfrak{t}_+^* \cap
\operatorname{Ad}^\vee(K)(i\mathfrak{k}_2^\perp)
\\
&=\left\{(\lambda_1,\ldots,\lambda_n)\in i\mathfrak
t^*_+:\,\lambda_{2i-1}=\lambda_i \ \mathrm{for}\,
 1\leq i\leq\left[\frac{n}{2}\right] \right\}.
\end{align*}
On the other hand,
it follows from Lemma \ref{lem:Kdeco} that
 the asymptotic $K$-support of $\pi_{i\lambda,\delta}$ amounts to
$$
AS_K(\pi_{i\lambda,\delta}) = \R_+(1,0,\dots,0).
$$
Hence, the triple $(G,G_2,\pi_{i\lambda,\delta})$ satisfies
\begin{equation}
\label{eqn:criterion}
AS_K(\pi_{i\lambda,\delta}) \cap
\operatorname{Ad}^\vee(K)(i\mathfrak{k}_2^\perp)
= \{0\}.
\end{equation}
This is nothing but the criterion for discrete decomposability of the
restriction of the unitary representation
$\pi_{i\lambda,\delta}|_{G_2}$
(\cite[Theorem 2,9]{xkAnn98}).

For $G_1=Sp(n,\R)$,
 we saw in Theorem \ref{thm:2}
 that the restriction $\pi_{i \lambda, \delta}^{GL(2n,\R)}|_{G_1}$
 stays irreducible.
Thus, this is another (obvious) example of discretely decomposable
 branching law.
We can see this fact directly from the observation
 that $G_1$ and $G_2$ have the same maximal compact subgroups,
$$
     (K_1:=)K \cap G_1 = K \cap G_2 (=:K_2).
$$
In fact, we get from \eqref{eqn:criterion}
$$
     AS_K(\pi_{i \lambda, \delta}) \cap \operatorname{Ad}^\vee(K)(i{\mathfrak {k}}_1^{\perp})=\{0\}.
$$
Therefore, the restriction $\pi_{i \lambda, \delta}|_{G_1}$
 is discretely decomposable, too.

\begin{rem}
In contrast to the restriction of the
quantization of elliptic orbits (equivalently, of Zuckerman's
$A_{\mathfrak q}(\lambda)$-modules), it is rare that the
restriction of the quantization of hyperbolic orbits (equivalently,
unitarily induced representations from real parabolic subgroups) is discretely
decomposable with respect to non-compact reductive subgroups.
Another discretely decomposable case was found  by Lee--Loke
in their study of the Jordan--H\"{o}lder series of a certain
degenerate principal series representations.
\end{rem}

\section{Branching law for $GL(N,\R)\downarrow GL(p,\R)\times
GL(q,\R)$} \label{sec:glpq}

\hfill\break Let $N=p+q$ $(p,q \ge 1)$, and consider a
subgroup $G_3 := GL(p,\R) \times GL(q,\R)$ in $G := GL(N,\R)$.
The restriction of $\pi_{i\lambda,\delta}^{GL(N,\R)}$ with respect to
the symmetric pair
\[
(G,G_3) = (GL(N,\R), GL(p,\R) \times GL(q,\R))
\]
is decomposed into the same family of degenerate principal series
representations of $G_3$:
\begin{thm}[Branching law $GL(p+q,\R)\downarrow GL(p,\R)\times GL(q,\R)$]
\label{thm:6.1}
$$
\pi_{i\lambda,\delta}^{GL(p+q,\R)}|_{G_3}\simeq
\sum_{\delta'=0,1}\int_\R^\oplus \pi_{i\lambda',\delta'}^{GL(p,\R)}\boxtimes \pi_{i(\lambda-\lambda'),
\delta-\delta'}^{GL(q,\R)}d\lambda'.
$$
\end{thm}
\begin{proof}[Outline of Proof]
The proof is similar to that of Theorem \ref{thm:3}.
The group $G_3=GL(p,\R)\times GL(q,\R)$ acts on
$\mathbb P^{p+q-1}\R$ with an open dense orbit $\mathcal O_3$ which
has a $G_3$-equivariant fibration
$$
\R^\times\to\mathcal O_3\to\left(GL(p,\R)/P_p)\times (GL(q,\R)/P_q\right).
$$

Hence, taking the Mellin transform by the $\R^\times$-action along the
fiber,
we get Theorem \ref{thm:6.1}.
\end{proof}

\section{Branching law for $GL(N,\R)\downarrow O(p,q)$}\label{sec:opq}
\hfill\break For $N=p+q$, we introduce the standard quadratic form
of signature $(p,q)$ by
\[
Q(x) :=
x_1^2+\dots+x_p^2-x_{p+1}^2-\dots-x_{p+q}^2
\quad\text{for $x \in \R^{p+q}$}.
\]
Let $G_4$ be the indefinite
orthogonal group defined by
\[
O(p,q) :=
\{ g \in GL(N,\R): Q(gx) = Q(x) \quad\text{for any $x\in\R^{p+q}$} \}.
\]
For $q=0$,
$G_4$ is nothing but a maximal compact subgroup $K=O(N)$ of $G$,
and the branching law
$\pi_{i\lambda,\delta}^{GL(N,\R)}|_{G_4}$ is so called the $K$-type
formula.

In order to describe the branching law $G \downarrow G_4$ for general
$p$ and $q$,
we introduce a family of irreducible unitary representations
of $G_4$,
to be denoted by $\pi_{+,\nu}^{O(p,q)}$ $(\nu\in A_+(p,q)$ below),
$\pi_{-,\nu}^{O(p,q)}$ $(\nu\in A_+(q,p))$,
and $\pi_{i\nu,\delta}^{O(p,q)}$ $(\nu\in\mathbb{R})$ as follows.
Let $\mathfrak{t}$ be a compact Cartan
subalgebra of $\mathfrak{g}_4$,
and we take a standard dual basis $\{ e_j\}$ of $\mathfrak{t}$  such that
the set of roots for
$\mathfrak{k}_4 := \mathfrak{o}(p) \oplus \mathfrak{o}(q)$
is given by
\begin{align*}
\Delta(\mathfrak{k}_4,\mathfrak{t}_4)
={}& \{ \pm(e_i\pm e_j) :
     1 \le i < j \le [\frac{p}{2}] \text{ \ or \ }
     [\frac{p}{2}]+1 \le i < j \le [\frac{p}{2}]+[\frac{q}{2}] \}
\\
&
  \cup
  \{ \pm e_i : 1 \le i \le [\frac{p}{2}] \}
  \ (p: \text{odd})
\\
&
  \cup
  \{ \pm e_i : [\frac{p}{2}]+1 \le i \le [\frac{p}{2}]+[\frac{q}{2}] \}
  \ (q: \text{odd}).
\end{align*}
Then, attached to the coadjoint orbits
$\operatorname{Ad}^\vee(G_4)(\nu e_i)$ for $\nu \in A_+(p,q)$
and
$\operatorname{Ad}^\vee(G_4)(\nu e_{[\frac{p}{2}]+1})$ for
$\nu \in A_+(q,p)$,
we can define unitary representations of $G_4$,
to be denoted by $\pi_{+,\nu}^{O(p,q)}$ and $\pi_{-,\nu}^{O(p,q)}$ as their
geometric quantizations.
These representations are
realized in Dolbeault cohomologies over the corresponding coadjoint
orbits endowed with $G_4$-invariant complex structures,
and their underlying $(\mathfrak{g}_{\mathbb{C}},K)$-modules are
obtained also as
 cohomologically induced representations from
characters of certain $\theta$-stable parabolic subalgebras
 (see \cite[\S5]{KobOr2} for details).

We normalize $\pi_{+,\nu}^{O(p,q)}$ such that
its infinitesimal character is given by
\[
(\nu, \frac{p+q}{2}-2, \frac{p+q}{2}-3, \dots, \frac{p+q}{2} -
 [\frac{p+q}{2}])
\]
in the Harish-Chandra parametrization.
The parameter set that we need for $\pi_{+,\nu}^{O(p,q)}$ is
$A_+(p,q):=A_+^0(p,q)\cup A_+^1(p,q)$ where
$$
A_+^\delta(p,q):=\left\{\begin{array}{ll}
                         \{\nu\in2\mathbb Z+\frac{p-q}2+1+\delta:\nu>0\},& (p>1,q\neq0); \\
                          \{\nu\in2\mathbb Z+\frac{p-q}2+1+\delta:\nu>\frac p2-1\},& (p>1,q=0); \\
                         \varnothing,&
                         (p=1,(q,\delta)\neq(0,1))\\
                         &
                         or\, (p=0);\\
                         \{\frac12\},& (p=1, (q,\delta)=(0,1)).
                       \end{array}\right.
$$
Notice that the identification $O(p,q)\simeq O(q,p)$ induces the equivalence $\pi_{-,\nu}^{O(p,q)}\simeq \pi_{+,\nu}^{O(q,p)}$.

For $p,q>0$ the group $G_4=O(p,q)$ is non-compact ant there are continuously many hyperbolic
coadjoint orbits.
Attached to (minimal) hyperbolic coadjoint orbits,
we can define another family of irreducible unitary representations of
$G_4$, to be denoted by $\pi_{i\nu,\delta}^{O(p,q)}$
for $\nu\in \R$ and $\delta \in \{0,1\}$.
Namely, let
 $\pi_{i\nu,\delta}^{O(p,q)}$ be the unitary representation of
$G_4$ induced from a unitary character $(i\nu,\delta)$ of a
maximal parabolic subgroup of $G_4$ whose
Levi part is $O(1,1)\times O(p-1,q-1)$.

We note that the Knapp--Stein intertwining operator gives a unitary
isomorphism
\[
\pi_{i\nu,\delta}^{O(p,q)} \simeq \pi_{-i\nu,\delta}^{O(p,q)}
\quad
(\nu \in \R, \  \delta = 0,1).
\]

\begin{thm}[Branching law $GL(p+q,\R)\downarrow O(p,q)$]
\label{thm:5}
$$
\pi_{i\lambda,\delta}^{GL(p+q,\R)}|_{O(p,q)}\simeq
\sideset{}{^\oplus}\sum_{\nu \in A_+^\delta(p,q)} \pi_{+,\nu}^{O(p,q)}\oplus
\sideset{}{^\oplus}\sum_{\nu \in A_+^\delta(q,p)}\pi_{-,\nu}^{O(p,q)}\oplus
2\int_{\R_+}^\oplus\pi_{i\nu,\delta}^{O(p,q)}d\nu.
$$
\end{thm}

Notice that in case when $q=0$ the latter two components of the above decomposition
do not occur and one gets the $K$-type formula $GL(n,\R)\downarrow O(n)$.

As a preparation of the proof,
we formalize the Plancherel formula on the hyperboloid from a modern
viewpoint of representation theory.

Let $X(p,q)_\pm$ be a hypersurface in $\R^{p+q}$ defined by
$$
X(p,q)_\pm:=\{x=(x',x'')\in\R^{p+q}\,:\,|x'|^2-|x''|^2=\pm1\}.
$$

We endow $X(p,q)_\pm$ with pseudo-Riemannian structures by restricting
$ds^2 = dx_1^2 + \dots + dx_p^2 - dx_{p+1}^2 - \dots - dx_{p+q}^2$
on $\R^{p+q}$.
Then,
$X(p,q)_{\pm}$ becomes a {\it{space form}}
 of pseudo-Riemannian manifolds
 in the sense that its sectional curvature $\kappa$
 is constant.
To be explicit,
 $X(p,q)_+$ has a pseudo-Riemannian structure
 of signature $(p-1,q)$
 with sectional curvature $\kappa \equiv 1$,
 whereas $X(p,q)_-$
 has a signature $(p,q-1)$
 with $\kappa \equiv -1$.
Clearly,
 $G_4$ acts on $X(p,q)_{\pm}$
 as isometries.

We denote by $L^2(X(p,q)_\pm)$ the Hilbert space consisting of
square integrable functions on $X(p,q)_\pm$ with respect to the
induced measure from $ds^2|_{X(p,q)}$.

The irreducible decomposition of the unitary representation of $G_4$ on
$L^2(X(p,q)_\pm)$ is equivalent to the spectral decomposition
of the Laplace--Beltrami operator on $X(p,q)_\pm$ with respect to the
$G_4$-invariant pseudo-Riemannian structures.
The latter viewpoint was established by Faraut \cite{xFaraut} and Strichartz \cite{xSt}.

As we saw in \cite[\S 5]{KobOr2},
 the discrete series representations on hyperboloids $X(p,q)_\pm$ are
 isomorphic to $\pi_{\pm,\nu}^{O(p,q)}$
 with parameter set $A_{\pm}(p,q)$.
\begin{eqnarray}\label{eqn:G4bdl}
  L^2(X(p,q)_+)_\delta &=& \sum_{\nu\in A_+^\delta(p,q)}\pi_{+,\nu}^{O(p,q)}\oplus\int_{\R_+}^\oplus\pi_{i\nu,\delta}^{O(p,q)}d\nu, \\
    L^2(X(p,q)_-)_\delta &=& \sum_{\nu\in A_+^\delta(q,p)}\pi_{-,\nu}^{O(p,q)}\oplus\int_{\R_+}^\oplus\pi_{i\nu,\delta}^{O(p,q)}d\nu.
\end{eqnarray}
Here we note that each irreducible decomposition is multiplicity
free, the continuous spectra in both decompositions are the same and
the discrete ones are distinct.

\begin{proof}[Proof of Theorem \ref{thm:5}]
According to the decomposition
\[
\R^{p+q} \underset{\text{dense}}{\supset}
\{ x \in \R^{p+q} : Q(x) > 0 \}
\cup
\{ x \in \R^{p+q} : Q(x) < 0 \},
\]
the group $G_4=O(p,q)$ acts on $\mathbb P^{p+q-1}\R$ with
two open orbits, denoted by $\mathcal O_4^+$
 and $\mathcal O_4^-$.
A distinguishing feature for $G_4$ is that these open $G_4$-orbits
 are reductive homogeneous
spaces.
To be explicit,
let $H_4^+$ and $H_4^-$ be the isotropy subgroups of $G_4$ at
$[e_1] \in \mathcal{O}_4^+$
and
$[e_{p+q}] \in \mathcal{O}_4^-$,
respectively,
where $\{e_j\}$ denotes the standard basis of $\R^{p+q}$.
Then we have
\begin{align*}
&\mathcal O_4^+\simeq G_4/H_4^+ = O(p,q)/(O(1)\times O(p-1,q)),
\\
&\mathcal O_4^-\simeq G_4/H_4^- = O(p,q)/(O(p,q-1)\times O(1)).
\end{align*}

Correspondingly,
the restriction of the line bundle
$\mathcal{L}_{i\lambda,\delta} = G \times_P \chi_{i\lambda,\delta}$
to the open sets $\mathcal{O}_4^\pm$ of the base space $G/P$ is given by
$$
G_4 \times_{H_4^\pm} \C_\delta,
$$
where $\C_\delta$ is a one-dimensional representation of $H_4^\pm$
defined by
\begin{alignat*}{2}
&O(1) \times O(p-1,q) \to \C^\times,
&(a,A) \mapsto a^\delta,
\\
&O(p,q-1) \times O(1) \to \C^\times,
\
&(B,b) \mapsto b^\delta,
\end{alignat*}
respectively. It is noteworthy that unlike the cases $G_2=GL(n,\C)$ and $G_3=GL(p,\R)\times GL(q,\R)$, the
continuous parameter $\lambda$ is not involved in (\ref{eqn:G4bdl}).

Since the union $\mathcal O_4^+\cup\mathcal O_4^-$ is open dense in
$\mathbb P^{p+q-1}\R$, we have a $G_4$-unitary equivalence
(independent of $\lambda$):
$$
\mathcal H_{i\lambda,\delta}^{GL(p+q,\R)}|_{G_4}\simeq
L^2(G_4 \times_{H_4} \C_\delta, \mathcal{O}_4^+)
\oplus L^2(G_4 \times_{H_4} \C_\delta, \mathcal{O}_4^-).
$$

Sections for the line bundle
$G_4 \times_{H_4^\pm} \C_\delta$ over $\mathcal{O}_4^\pm$ are
identified with even functions $(\delta=0)$ or odd functions
$(\delta=1)$ on
hyperboloids $X(p,q)_\pm$ because
$X(p,q)_\pm$ are double covering manifolds of
$\mathcal{O}_4^\pm$.

According to the parity of functions on the hyperboloid $X(p,q)_\pm$,
we decompose
$$
L^2(X(p,q)_\pm)
= L^2(X(p,q)_\pm)_0 \oplus L^2(X(p,q)_\pm)_1.
$$

Hence, we get Theorem \ref{thm:5}.
\end{proof}

\section{Tensor products $\mathrm{Met}^\vee \otimes
\mathrm{Met}$}\label{sec:75} \hfill\break The irreducible
decomposition
 of the tensor product of two representations
 is a special example
 of branching laws.
It is well-understood that the tensor product of the same
Segal--Shale--Weil representation 
(e.g.\ $\mathrm{Met} \otimes \mathrm{Met}$)
decomposes into a \emph{discrete} direct sum of lowest weight
representations of $Sp(n,\R)$ (see \cite{KV}).
In this section,
we prove:
\begin{thm}\label{thm:MetMet}
Let $\mathrm{Met}$ be the Segal--Shale--Weil representation of
the metaplectic group $Mp(n,\R)$, and $\mathrm{Met}^\vee$ its
contragredient representation. Then the tensor product
representation $\mathrm{Met}^\vee \otimes \mathrm{Met}$ is well-defined
as a representation of $Sp(n,\R)$, and decomposes into 
the direct integral of
irreducible unitary
representations as follows:
\begin{equation}\label{eqn:mettensor}
\mathrm{Met}^\vee \otimes \mathrm{Met}
\simeq \sum_{\delta=0,1}
\int_{\R_+}^\oplus 2\pi_{i\lambda,\delta}^{Sp(n,\R)} d\lambda.
\end{equation}
\end{thm}
\begin{rem} The branching formula in Theorem \ref{thm:MetMet} may be regarded as the dual pair correspondence
$O(1,1)\cdot Sp(n,\R)$ with respect to the Segal--Shale--Weil representation of $Mp(2n,\R)$. We note that the
Lie group $O(1,1)$ is non-abelian, and its finite dimensional irreducible unitary representations are generically of
dimension two, which corresponds the multiplicity two in the right-hand side of (\ref{eqn:mettensor}).
\end{rem}

\begin{proof}
By Fact \ref{fact:irred},
the Weyl operatorcalculus
\begin{equation}\label{eqn:Opiso}
\mathrm{Op} : L^2(\R^{2n})
\stackrel{\sim}{\to}
\operatorname{ HS}(L^2(\R^n), L^2(\R^n))
\end{equation}
gives an intertwining operator as unitary representations
of $Mp(n,\R)$.
We write $L^2(\R^n)^\vee$ for the dual Hilbert space,
and identify
\begin{equation}\label{eqn:HSdual}
\operatorname{ HS}(L^2(\R^n), L^2(\R^n))
\simeq L^2(\R^n)^\vee \widehat{\otimes} L^2(\R^n),
\end{equation}
where $\widehat{\otimes}$ denotes the completion of the tensor product
of Hilbert spaces.
Composing \eqref{eqn:Opiso} and \eqref{eqn:HSdual},
we see that the tensor product representation
$\mathrm{Met}^\vee \otimes \mathrm{Met}$ of $Mp(n,\R)$ is unitarily
equivalent to the regular representation on $L^2(\R^{2n})$.
This representation on the phase space $L^2(\R^{2n})$ is well-defined as a
representation of $Sp(n,\R)$.

We consider
the Mellin transform on $\R^{2n}$, which is defined as the Fourier transform
along the radial direction:
$$
f\to\frac1{4\pi}\int_{-\infty}^\infty |t|^{n-1+i\lambda}({\mathrm{ sgn}}t)^\delta f(tX)dt,
$$
with $\lambda\in\R,\delta=0,1,X\in\R^{2n}$.
Then, the Mellin transform gives a spectral
decomposition of the Hilbert space $L^2(\R^{2n})$.
Therefore, the phase space representation
$L^2(\R^{2n})$ is decomposed as a direct integral of Hilbert spaces:
\begin{equation}\label{eqn:Mell} L^2(\R^{2n})\simeq
\sum_{\delta=0,1}\int_\R^\oplus V_{i\lambda,\delta}\,
d\lambda.
\end{equation}
Since
$\pi_{i\lambda,\delta}^{Sp(n,\R)} \simeq \pi_{-i\lambda,\delta}^{Sp(n,\R)}$
(see \eqref{eqn:Spdual}), we get Theorem \ref{thm:MetMet}.
\end{proof}

\noindent \textbf{Acknowledgement.}
     The authors are grateful to the Institut des Hautes \'Etudes Scientifiques, the Institute for the Physics
     and Mathematics of the Universe of the Tokyo University, the Universities of {\AA}rhus and Reims where
     this work was done.

\bigskip
\footnotesize{ \noindent Addresses:
(TK) Graduate School of Mathematical Sciences, IPMU, The University of
 Tokyo,
3-8-1 Komaba, Meguro, Tokyo, 153-8914 Japan;
Institut des Hautes \'{E}tudes Scientifiques,
Bures-sur-Yvette, France (current address).
\\
(B\O ) Matematisk Institut, Byg.\,430, Ny Munkegade, 8000 Aarhus C,
Denmark.\\
(MP)
Laboratoire de Math\'ematiques, Universit\'e
de Reims, 51687 Reims, France.\bigskip

\noindent \texttt{{
 toshi@ms.u-tokyo.ac.jp,
 orsted@imf.au.dk,
 pevzner@univ-reims.fr.}}

}

\end{document}